\documentclass[12pt]{article}
\usepackage{amssymb,amsmath,amsthm,amsfonts,enumerate, bbm, mathdots}

\newtheorem{thm}{Theorem}

\newtheorem{prop}[thm]{Proposition}
\newtheorem*{mth}{Main theorem}
\newtheorem{rmk}{Remark}
\newtheorem{lem}[thm]{Lemma}
\newtheorem{cor}[thm]{Corollary}
\newtheorem{q}[thm]{Question}

\theoremstyle{example}

\newcommand{\mH}{\mathfrak{H}}

\def\N{\mathbb {N}} 

\def\Q{\mathbb {Q}}
\def\R{\mathbb {R}}

\def\C{\mathbb {C}}
\def\h{\hbar}

\def\Hht{\mH_{\h, t}}

\def\c{\cdot}
\def\cs{\cdots}
\def\ls{\ldots} 
\def\f{\frac} 
\def\d{\displaystyle}
\def\k{{\mathbbm k}} 
\def\p{{\mathbbm p}} 
\def\z{\zeta}
\def\s{\zeta^{\star}}
\def\q{\zeta_{q}}
\def\qs{\zeta^{\star}_{q}}
\def\t{\zeta^{t}}
\def\tq{\zeta^{t}_{q}}

\def\Dq{\mathcal{D}_q}
\def\Tq{\Theta_q}

\title{Sum of interpolated multiple $q$-zeta values}
\date{\empty}
\author{Zhonghua Li \thanks{School of Mathematical Sciences, Tongji University, Shanghai, 200092
China\newline
E-mail address : zhonghua\_li@tongji.edu.cn
} 
\and 
Noriko Wakabayashi \thanks{College of Science and Engineering, Ritsumeikan University, 1-1-1, Nojihigashi, Kusatsu-city, Shiga 525-8577, Japan\newline
E-mail address : noriko-w@fc.ritsumei.ac.jp}
}

\begin{document} 
\maketitle
\begin{abstract}
Interpolated multiple $q$-zeta values are deformation of multiple $q$-zeta values with one parameter, $t$, and restore classical multiple zeta values as $t = 0$ and $q \to 1$.  
In this paper, we discuss generating functions for sum of interpolated multiple $q$-zeta values with fixed weight, depth and $i$-height. The functions are systematically expressed in terms of the basic hypergeometric functions. 
%
Compared with the result of Ohno and Zagier, our result includes three generalizations: general height, $q$-deformation and $t$-interpolation. 
As an application, we prove some expected relations for interpolated multiple $q$-zeta values including sum formulas. 
\end{abstract}
\tableofcontents
\section{Introduction/Main theorem}\label{Intro}
For any multi-index $\k = (k_1, k_2, \ls, k_l) \in \N^l$, 
the quantities 
wt$(\k) = k_1 +k_2 + \cs +k_l$, dep$(\k)=l$ and $i$-ht$(\k) = \sharp\{j | k_j \geq i+1\}$ 
are called weight, depth and $i$-height of $\k$, respectively. 

For a formal parameter $q$ and an index $\k=(k_1, k_2, \ls, k_l)$ of positive integers with $k_1 \geq 2$ (which we call admissible), $q$-analogues of multiple zeta and zeta-star values ($q$MZVs and $q$MZSVs for short, respectively) are defined by 
\begin{align*}
\q(\k) = \q(k_1, k_2, \ls, k_l) & = \sum_{m_1 > m_2 > \cs > m_l \geq 1} \f{q^{(k_1-1)m_1 + (k_2-1)m_2 + \cs +(k_l-1)m_l}}{{[m_1]}^{k_1}{[m_2]}^{k_2} \cs {[m_l]}^{k_l}} \ (\in \Q[[q]]),\\
\qs(\k) = \qs(k_1, k_2, \ls, k_l) & = \sum_{m_1 \geq m_2 \geq \cs \geq m_l \geq 1} \f{q^{(k_1-1)m_1 + (k_2-1)m_2 + \cs +(k_l-1)m_l}}{{[m_1]}^{k_1}{[m_2]}^{k_2} \cs {[m_l]}^{k_l}} \ (\in \Q[[q]]), 
\end{align*}
where $[n]$ denotes the $q$-integer $\d [n] =\f{1-q^n}{1-q}$. 
In the case of $l=1$, $q$MZVs and $q$MZSVs coincide and are reduced to 
$$\q(k)=\sum_{m \geq 1}\f{q^{(k-1)m}}{[m]^k}. $$ 
The $q$MZ(S)Vs are investigated for example in Bradley\cite{B}, Okuda-Takeyama\cite{OT} and Zhao\cite{Z}. 
If $q \in \C$, $q$MZVs and $q$MZSVs are absolutely convergent in $|q|<1$. 
Taking the limit as $q \to 1$, $q$MZ(S)Vs turn into ordinary multiple zeta(-star) values (MZ(S)Vs for short) given by
\begin{align*}
\z(\k)=\z(k_1, k_2, \ls, k_l) & = \sum_{m_1 > m_2 > \cs > m_l \geq 1} \f{1}{{m_1}^{k_1}{m_2}^{k_2} \cs {m_l}^{k_l}} \ (\in \R),\\
\s(\k)=\s(k_1, k_2, \ls, k_l) & = \sum_{m_1 \geq m_2 \geq \cs \geq m_l \geq 1} \f{1}{{m_1}^{k_1}{m_2}^{k_2} \cs {m_l}^{k_l}} \ (\in \R). 
\end{align*} 
The generating functions of sum of MZ(S)Vs are studied in \cite{AKO, AO, AOW, Li2, OZ}. 
Those for $q$MZ(S)Vs are studied in \cite{B, Li, OT, OO, Take}. 
In this paper, we give a universal result including these as special cases. 

As is introduced in \cite{W},  
with an additional parameter, $t$, we define the interpolation of $q$MZVs (abbreviated as $t$-$q$MZVs) by 
\begin{align*}
\tq(\k)=\tq(k_1, k_2, \ls, k_l) = {\sum_{\p}} (1-q)^{k-\rm{wt}(\p)} \z_q({\p}) t^{l - \rm{dep}(\p)} \ (\in \Q[[q]][t]),  
\end{align*}
where $k=k_1+k_2+\cdots+k_l$ 
and $\sum_{\p}$ stands for the sum where 
$\p$ runs over all indices of the form  
$\p=(k_1\ \square\ \cs\ \square\ k_l)$ in which each $\square$ is filled by three candidates:
$``,"$, $``+"$ or $``-1+"$ (minus $1$ plus). 
If $q \in \C$, $t$-$q$MZVs are absolutely convergent in $|q|<1$. 
Taking the limit as $q \to 1$, $t$-$q$MZVs turn into $t$-MZVs which were introduced in Yamamoto\cite{Y}. 
We notice that $\q^0=\q$, $\q^1=\qs$ and $\tq(k) = \q(k)$ for $k \geq 2$ and arbitrary $t$. 

The $q$-shifted factorial $(a; q)_n$ is defined by
$$
(a; q)_n = 
\left\{ \begin{array}{ll}
1, & n=0,  \\
(1-a) (1-aq) \cs (1-aq^{n-1}), & n= 1, 2, \ls .
\end{array}\right.
$$
For simplicity,  we denote the product $(a_1; q)_n (a_2; q)_n \cs (a_m; q)_n$ by $(a_1, a_2, \ls, a_m; q)_n$. 
For a positive integer $r$,  the basic hypergeometric function 
$\d {}_{r+1}\phi_{r}
\left[ 
\begin{array}{c}
a_1,  \ls, a_{r+1}\\
b_1,  \ls, b_r
\end{array} 
; q, z \right]$  
is defined by the series
$$
{}_{r+1}\phi_{r}
\left[ 
\begin{array}{c}
a_1, \ls, a_{r+1}\\
b_1, \ls, b_r
\end{array} 
; q, z \right]
= \sum_{n=0}^{\infty} \f{(a_1, \ls, a_{r+1}; q)_n}{(q, b_1, \ls, b_r; q)_n}z^n, $$
where $b_j \neq q^{-m}$ for $m=0, 1, \ls$ and $j=1, 2, \ls, r$. 
If $0< q < 1$, the series converges absolutely in $|z| < 1$. 

Let $r$ be a positive integer. 
For any non-negative integers $k, l, h_1, \ls, h_r$, 
we define 
\begin{align*}
\xi^t_0&(k, l, h_1, \ls, h_r) = \sum_{\k \in I_0(k, l, h_1, \ls, h_r)} \tq(\k), 
\end{align*}
where $I_0(k, l, h_1, \ls, h_r)$ denotes the subset of admissible multi-indices (to wit indices with additional requirement $k_1 \geq 2$) of weight $k$, depth $l$, $1$-height $h_1$, \ls, $r$-height $h_r$. In the definition, the sum is treated as $0$ whenever the index set  is empty. 
We also define
\begin{align*}
\Psi^t_0& =\Psi^t_0(u_1, u_2, \ls, u_{r+2}) \\
              & = \sum_{k, l, h_1, \ls, h_r \geq 0} \xi^t_0(k, l, h_1, \ls, h_r)
                     u_1^{k-l-\sum_{j=1}^{r}h_j} u_2^{l-h_1} u_3^{h_1 - h_2} \cs u_{r+1}^{h_{r-1} - h_{r}} u_{r+2}^{h_r}. 
\end{align*}

In order to state our main theorem, we give some other notations below. 
For variables $u_1, u_2, \ls, u_{r+2}$, we put
\begin{align}\label{x1}
x_1=\f{u_1}{1+(1-q)u_1}
\end{align}
and
\begin{align}\label{xj}
x_j=\f{1}{u_1^{r+2-j}} \left\{ \sum_{k=j}^{r+1} {k-2 \choose j-2} ( (q-1)u_1)^{k-j} (u_1^{r+2-k} u_k- u_{r+2}) 
+ \f{u_{r+2}}{( 1+(1-q)u_1)^{j-1}} \right\}
\end{align}
for $j=2, \ls, r+2$. 
Let $\alpha_1, \alpha_2, \ls, \alpha_{r+1}$ be parameters determined by 
$$
\left\{
\begin{array}{l}
\alpha_1 + \alpha_2+ \cs + \alpha_{r+1} = (1-t)x_2-x_1,\\
\d\sum_{1 \leq i_1 < \cs < i_j \leq r+1} \alpha_{i_1} \cs \alpha_{i_j} =(1-t) (x_{j+1} -x_1x_j), \quad j=2, \ls, r+1,  
\end{array}
\right.
$$
and put
\begin{align}\label{ai}
a_i = \f{q}{1+(1-q)\alpha_i}, \quad i=1, 2, \ls, r+1. 
\end{align}
Also let $\beta_1, \beta_2, \ls, \beta_{r+1}$ be parameters determined by 
$$
\left\{
\begin{array}{l}
\beta_1 + \beta_2+ \cs + \beta_{r+1} = -(x_1+t x_2),\\
\d\sum_{1 \leq i_1 < \cs < i_j \leq r+1} \beta_{i_1} \cs \beta_{i_j} =-t (x_{j+1} -x_1x_j), \quad j=2, \ls, r+1,  
\end{array}
\right.
$$
and put 
\begin{align}\label{bi}
b_i = \f{q^2}{1+(1-q)\beta_i}, \quad i=1, 2, \ls, r+1. 
\end{align}
Then we state our main theorem, 
in which one can find that the function $\Psi^t_0$ is expressed in terms of the basic hypergeometric functions. 
\begin{mth}\label{mth}
Let $r$ be a positive integer and $a_1, a_2, \ls, a_{r+1}$, $b_1, b_2, \ls, b_{r+1}$ as above. 
Then we obtain
\begin{align*}
\Psi^t_0 & = \f{(1+(1-q)u_1)^2}{1-qu_1 - t (1-qu_1) \sum_{k=2}^{r+1}q^{k-2} u_k -tq^r u_{r+2}} \\
 & \times \sum_{j=0}^{r-1} A_j \widetilde{B_j} \ {}_{r+2} \phi _{r+1} 
 \left[ 
 \begin{array}{c}
 q^{j+1}, a_1q^j, \ls, a_{r+1}q^j \\
 b_1q^j, \ls, b_{r+1}q^j
 \end{array}
 ; q, \f{b_1 \cs b_{r+1}}{q^r a_1 \cs a_{r+1}}
 \right], 
\end{align*}
where 
$$A_j = \sum_{m=j}^{r-1} c_m S_q(m+1, j+1) + \f{u_1u_2}{(1+(1-q)u_1)^2} S_q(r, j+1), \quad j=0, \ls, r-1, $$
with
\begin{align*}
{c}_m = & \sum\limits_{k=r-m+1}^{r+1}\left\{\left({k-2\atop
r-m}\right)+\frac{(1-q)u_1}{1+(1-q)u_1}\left({k-2\atop
r-m-1}\right)\right\}(-(1-q))^{k-r+m-2}\\
& \times \left(\frac{u_k}{1+(1-q)u_1}-u_1^{k-r-2}u_{r+2}
+\frac{(1-q)u_{k+1}}{1+(1-q)u_1}\right),    
\end{align*}
$S_q(n, k)$ is the $q$-Stirling number of the second kind which is defined recursively by 
$$S_q(n, k) = 
\left\{ \begin{array}{ll}
q^{k-1} S_q(n-1, k-1) + [k] S_q(n-1, k), & 0 < k \leq n, \\
1, & n = k= 0, \\
0, & \text{otherwise},  
\end{array}\right. $$
and
$$\widetilde{B_j} = q \c \f{(q, a_1, \ls, a_{r+1}; q)_j}{(b_1, \ls, b_{r+1}; q)_j} \left( \f{b_1 \cs b_{r+1}}{q^r (1-q) a_1 \cs a_{r+1}} \right)^j, \quad j=0, \ls, r-1. $$
\end{mth}
Before giving its proof, we discuss in the next section some special cases of the theorem to find that several known results are implied by this theorem.  
In \S \ref{tqMPL}, we introduce interpolation of $q$-analogue of multiple polylogarithms ($t$-$q$MPLs for short), 
investigate several difference formulas to obtain the relation between $t$-$q$MPLs and the basic hypergeometric functions, 
and show that their values at $z=q$ can be written as a linear combination of $t$-$q$MZVs. 
Thanks to the results in  \S \ref{tqMPL}, we establish our proof of the main theorem in \S \ref{SumtqMZV}. 

\section{Special cases of main theorem}\label{Special}
\subsection{The case of $r = 1$}
Our main theorem mentioned in the previous section is reduced to the following corollary in the case of $r=1$. 
\begin{cor}\label{cor}
We obtain the generating function of $t$-$q$MZVs for $1$-ht:  
\begin{align}\label{r1}
\Psi^t_0 = 
\f{q u_3}{(1-qu_1)(1-t u_2) - t qu_3} 
 {}_{3} \phi_{2} 
 \left[ 
 \begin{array}{c}
 q, a_1, a_2 \\
 b_1, b_2
 \end{array}
 ; q, \f{q\{ 1+(1-q)(1-t)u_2 \}}{1-(1-q)t u_2}
 \right],  
 \end{align}
where 
$$\quad a_i =\f{q}{1+(1-q)\alpha_i},\ b_i=\f{q^2}{1+(1-q)\beta_i}$$
and $\alpha_1, \alpha_2, \beta_1, \beta_2$ are determined by 
$$\left\{ \begin{array}{l}
\alpha_1+\alpha_2 
= \f{-u_1 +(1-t) \{u_2 + (1-q)(u_1u_2-u_3) \}}{1+(1-q)u_1}, \\
\alpha_1\alpha_2 
= \f{(1-t)(u_3 - u_1u_2)}{1+(1-q)u_1}, 
\end{array}
\right. 
\left\{ \begin{array}{l}
\beta_1+\beta_2 
= \f{-u_1 - t\{ u_2 + (1-q)(u_1 u_2 -u_3) \}}{1+(1-q)u_1}, \\
\beta_1\beta_2 
= \f{t(u_1u_2 - u_3)}{1+(1-q)u_1}. 
\end{array}
\right. $$
\end{cor}
\begin{proof}
Set $r=1$ in the main theorem. Since  
$A_0=\left( c_0 + \f{u_1 u_2}{(1+(1-q)u_1)^2} \right)S_q(1, 1)$, 
$\widetilde{B_0} = q$, 
$c_0 = - \f{u_1 u_2}{(1+(1-q)u_1)^2} + \f{u_3}{(1+(1-q)u_1)^2}$, 
$S_q(1, 1) =1$, 
we have 
$$A_0= \f{u_3}{(1+(1-q)u_1)^2}. $$
Hence we obtain
\begin{align}\label{r2}
\Psi^t_0 = \f{q u_3}{(1-q u_1)(1-t u_2) - t q u_3} 
{}_3\phi_2 \left[ \begin{array}{c} q, a_1, a_2 \\ b_1, b_2 \end{array} ; q, \f{b_1 b_2}{q a_1 a_2} \right], 
\end{align} 
where 
$$\quad a_i =\f{q}{1+(1-q)\alpha_i},\ b_i=\f{q^2}{1+(1-q)\beta_i}$$
and $\alpha_1, \alpha_2, \beta_1, \beta_2$ are determined by 
$$\left\{ \begin{array}{l}
\alpha_1+\alpha_2 = (1-t)x_2 - x_1, \\
\alpha_1\alpha_2 = (1-t)(x_3 - x_1x_2),  
\end{array}
\right. 
\left\{ \begin{array}{l}
\beta_1+\beta_2 = -(x_1+tx_2), \\
\beta_1\beta_2 = -t(x_3 - x_1x_2).  
\end{array}
\right. $$
Because of the transformation \eqref{x1} and \eqref{xj}, we conclude the corollary.  
\end{proof}
\subsubsection{Reduction to Okuda and Takeyama's result}
We find that our main theorem includes Okuda-Takeyama\cite[Theorem 3]{OT} by putting $t=0$ in Corollary \ref{cor}. 
In fact, 
%
if $t=0$ we have $b_1=q^2$ (since we may assume $\beta_1 = 0$ and $\beta_2 = -x_1$), 
and using Heine's summation formula (see \cite{G} for example)
\begin{align}\label{Heine}
{}_2 \phi_1 \left[ \begin{array}{c} a_1, a_2 \\ b_1 \end{array} ; q, \f{b_1}{a_1 a_2} \right] 
= \f{\left( \f{b_1}{a_1} ; q \right)_{\infty} \left( \f{b_1}{a_2} ; q \right)_{\infty}}{\left( b_1 ; q \right)_{\infty} \left( \f{b_1}{a_1 a_2} ; q\right)_{\infty}},
\end{align}
where $(a; q)_{\infty} = \d\prod_{n=1}^{\infty}(1-a q^{n-1})$, 
and the formula 
\begin{align}\label{log}
\log \prod_{n=1}^{\infty} \left(1-\f{q^n}{[n]} x \right) = \f{1}{q-1} \log\{1+ (1-q)x \} \sum_{n=1}^{\infty}\f{q^n}{[n]} - \sum_{n=2}^{\infty} \q(n) \sum_{m=0}^{\infty} \f{(q-1)^m}{m+n} x^{m+n}, 
\end{align}
we know that the equation \eqref{r2} turns into 
\begin{align*}
\Psi^0_0 & = 
\f{u_3}{u_3-u_1u_2} 
\left(
{}_{2} \phi_{1} 
\left[ 
\begin{array}{c}
\d\f{a_1}{q}, \d\f{a_2}{q} \\
\d\f{b_2}{q}
\end{array}
; q, \f{qb_2}{a_1a_2} \right] -1 \right)\\
& = \f{u_3}{u_3-u_1u_2} 
\left\{ \exp \left( \sum_{n=2}^{\infty} \q(n) \sum_{m=0}^{\infty} \f{(q-1)^m}{m+n}(u_1^{m+n} + u_2^{m+n} -z_1^{m+n} -z_2^{m+n})\right) -1  \right\}, 
\end{align*}
where  
$$\left\{ \begin{array}{l}
z_1+z_2 = u_1 + u_2 + (1-q)(u_1u_2-u_3), \\
z_1 z_2 =u_3.
\end{array} \right. $$
Moreover we find that this is reduced to the result in Ohno-Zagier\cite{OZ} as $q \to 1$.   
\subsubsection{Reduction to Takeyama's result}
Also our main theorem includes Takeyama\cite[Theorem 1.1]{Take} by putting $t=1$ in Corollary \ref{cor}. 
In fact, 
%
if $t=1$ we have $a_1 = q$ and $a_2 = q\{1+(1-q)u_1\}$ 
(since we may asumme $\alpha_1 = 0$ and $\alpha_2 = -x_1$). 
Hence the equation \eqref{r1} turns into
$$
\Psi^1_0 = 
\f{q u_3}{(1-qu_1)(1- u_2) -  qu_3} 
{}_{3} \phi_{2} 
\left[ 
\begin{array}{c}
q, q, q\{1+(1-q)u_1\} \\
\d\f{q^2}{x}, \f{q^2}{y}
\end{array}
; q, \f{q}{1-(1-q)u_2}
\right],  
$$
where $x= 1+(1-q) \beta_1, y=1+(1-q) \beta_2$. 
Therefore we have 
$$\left\{ \begin{array}{l}
x+y= 2+ (1-q)(\beta_1 + \beta_2) = \d\f{2+(1-q)(u_1-u_2) + (1-q)^2(u_3 - u_1 u_2)}{1+(1-q)u_1}, \\
xy=1+(1-q)(\beta_1 + \beta_2) + (1-q)^2\beta_1\beta_2 =\d \f{1-(1-q)u_2}{1+(1-q)u_1}.
\end{array} \right. 
$$
Moreover we find that this is reduced to the result in Aoki-Kombu-Ohno\cite{AKO} as $q \to 1$. 
\subsubsection{Sum formula for $t$-$q$MZVs I}
As an application of Corollary \ref{cor}, we establish the following sum formula for $t$-$q$MZVs.  
\begin{thm}\label{SF} For any positive integers $k > n \geq 1$, we have 
\begin{align*}
& \sum_{\text{wt}(\k)=k, \text{dep}(\k)=n \atop \k : \text{admissible}} \tq(\k) \\
& = \f{1}{k-1} \sum_{0 \leq l \leq j \leq n-1} {k-1 \choose j} {j \choose l} (k-1-l) t^j (1-t)^{n-1-j} (1-q)^{l} \q(k-l). 
\end{align*}
\end{thm}
\begin{rmk}
Taking the limit as $q \to 1$, this formula is reduced to the sum formula of $t$-MZVs
\begin{align*}
\sum_{\text{wt}(\k)=k, \text{dep}(\k)=n \atop \k : \text{admissible}} \t(\k) = \left\{ \sum_{j=0}^{n-1} {k-1 \choose j} t^j (1-t)^{n-1-j} \right\} \z(k) 
\end{align*}
which is proved in Yamamoto\cite{Y}. 
If $t=0$, we have the sum formula for $q$MZVs 
\begin{align*}
\sum_{\text{wt}(\k)=k, \text{dep}(\k)=n \atop \k : \text{admissible}} \q(\k) = \q(k) 
\end{align*}
which is proved in Bradley\cite{B}. 
And also, if $t=1$, we have  the sum formula for $q$MZSV
\begin{align*}
\sum_{\text{wt}(\k)=k, \text{dep}(\k)=n \atop \k : \text{admissible}} \qs(\k) = \f{1}{k-1} {k-1 \choose n-1}  \sum_{l=0}^{n-1} {n-1 \choose l} (k-1-l) (1-q)^l \q(k-l) 
\end{align*}
which is proved in Ohno-Okuda\cite{OO}. 
\end{rmk}
\begin{proof}[Proof of Theorem \ref{SF}]
Using the $q$-analogue of the Kummer-Thomae-Whipple formula (see \cite{G} for example)
\begin{align*}
{}_{3} \phi_{2} 
 \left[ \begin{array}{c} a_1, a_2, a_3 \\ b_1, b_2 \end{array} ; q, \f{b_1 b_2}{a_1 a_2 a_3} \right] 
= \f{\left(\f{b_1}{a_1}; q \right)_{\infty}\left(\f{b_1b_2}{a_2a_3} ; q \right)_{\infty}}{\left(b_1 ; q \right)_{\infty}\left( \f{b_1b_2}{a_1 a_2 a_3} ; q\right)_{\infty}}
{}_{3} \phi_{2} 
 \left[ \begin{array}{c} a_1, \f{b_2}{a_2}, \f{b_2}{a_3} \\ b_2, \f{b_1b_2}{a_2 a_3} \end{array} ; q, \f{b_1}{a_1} \right],  
\end{align*}
we have
\begin{align}
& {}_{3} \phi_{2} \nonumber
  \left[ \begin{array}{c} q, a_1, a_2 \\ b_1, b_2 \end{array} ; q, \f{q\{ 1+(1-q)(1-t)u_2 \}}{1-(1-q)t u_2} \right] \\ \nonumber
 & = {}_{3} \phi_{2} \left[ \begin{array}{c} q, a_1, a_2 \\ b_1, b_2 \end{array} ; q, \f{b_1 b_2}{q a_1 a_2} \right]\\ \nonumber
 & = \f{\left(\f{b_1}{q}; q \right)_{\infty}\left(\f{b_1b_2}{a_1a_2} ; q \right)_{\infty}}{\left(b_1 ; q \right)_{\infty}\left( \f{b_1b_2}{q a_1 a_2} ; q\right)_{\infty}}
{}_{3} \phi_{2} 
 \left[ \begin{array}{c} q, \f{b_2}{a_1}, \f{b_2}{a_2} \\ b_2, \f{b_1b_2}{a_1 a_2} \end{array} ; q, \f{b_1}{q} \right] \\
 & = \f{1-\f{b_1}{q}}{1-\f{b_1b_2}{q a_1 a_2}}
{}_{3} \phi_{2} \label{eq2}
 \left[ \begin{array}{c} q, \f{b_2}{a_1}, \f{b_2}{a_2} \\ b_2, \f{b_1b_2}{a_1 a_2} \end{array} ; q, \f{b_1}{q} \right]. 
\end{align}
Putting $u_3 = u_1 u_2$ in Corollary \ref{cor}, we have 
$$\alpha_1 = 0,\ \alpha_2 = \f{-u_1 + (1-t) u_2}{1+(1-q)u_1},\ a_1 = q,\ a_2 = \f{q \{1+(1-q)u_1\}}{1+(1-q)(1-t)u_2}, $$
$$\beta_1 = 0,\ \beta_2 = \f{-u_1-t u_2}{1+(1-q)u_1},\ b_1 = q^2,\ b_2 = \f{q^2 \{1+(1-q)u_1\}}{1-(1-q)t u_2}, $$
$$(1-qu_1)(1-tu_2) - tqu_3 = 1-qu_1 - t u_2. $$
Then by \eqref{eq2},  we get
$$\Psi^t_0 \nonumber
 = \f{q u_1 u_2}{1-q u_1 - t u_2} \c \f{1-q}{1-\f{b_2}{a_2}}
{}_{3} \phi_{2} 
 \left[ \begin{array}{c} q, \f{b_2}{q}, \f{b_2}{a_2} \\ b_2, \f{q b_2}{a_2} \end{array} ; q, q \right]. $$
By the definition of ${}_{3}\phi_{2}$, 
\begin{align}
\Psi^t_0 \nonumber
 & = \f{q(1-q)u_1 u_2}{(1-qu_1 - t u_2) \left( 1-\f{b_2}{a_2} \right)} \sum_{n=0}^{\infty} \f{1-\f{b_2}{q}}{1-b_2 q^{n-1}}\c \f{1-\f{b_2}{a_2}}{1-q^n \f{b_2}{a_2}} q^n \\ 
 & = \f{q(1-q)u_1 u_2 \left( 1-\f{b_2}{q} \right)}{(1-qu_1 - t u_2)} \sum_{n=0}^{\infty} \f{q^n}{(1-b_2 q^{n-1})\left(1-q^n \f{b_2}{a_2}\right)}. \label{eq4}
\end{align}
Since we have 
$$1-\f{b_2}{q} = \f{(1-q)(1-q u_1 - tu_2)}{1-(1-q)t u_2}$$
and 
\begin{align*}
& (1-b_2 q^{n-1})\left(1-q^n \f{b_2}{a_2}\right) \\
& = \f{\{ 1-q^{n+1} -(1-q)(t u_2 + q^{n+1} u_1) \} 
 \{ 1-q^{n+1} - (1-q) (t u_2 + q^{n+1}(1-t) u_2)\}}{\{ 1-(1-q)t u_2\}^2}, 
 \end{align*} 
we calculate the right-hand side of \eqref{eq4} as 
\begin{align}
\Psi^t_0 \nonumber
& = u_1 u_2 \{ 1-(1-q)t u_2 \}
\sum_{n=1}^{\infty} \f{q^n}{\{ [n] - (t u_2 + q^n u_1) \} \{ [n] - ( t + q^n (1-t)) u_2 \} } \\ \nonumber
& = u_1 u_2 \{ 1-(1-q)t u_2 \} \sum_{n=1}^{\infty} \f{q^n}{[n]^2} \sum_{m, l =0}^{\infty} \left( \f{t u_2 + q^n u_1}{[n]}\right)^m \left(\f{ t+ q^n (1-t)}{[n]} \right)^l u_2^l\\ \nonumber
& =  u_1 u_2 \{ 1-(1-q)t u_2 \} \\ \nonumber
& \quad \times \sum_{n=1}^{\infty} \f{q^n}{[n]^2} 
\sum_{i, j \geq 0} {i+j \choose i}\f{(tu_2)^i (q^n u_1)^j}{[n]^{i+j}}
\sum_{l, m \geq 0} {l+m \choose l}\f{t^l (q^n (1-t))^m}{[n]^{l+m}} u_2^{l+m}\\ \nonumber
& =  u_1 u_2 \{ 1-(1-q)t u_2 \} \\ 
& \quad \times \sum_{i, j, l, m \geq 0} {i+j \choose i} {l+m \choose l}  t^{i+l} (1-t)^m u_1^j u_2^{i+l+m}
\sum_{n=1}^{\infty}\f{q^{n(j+m+1)}}{[n]^{i+j+l+m+2}}. \label{eq5}
\end{align} 
Because of 
$$q^{nm} = \sum_{i=0}^l { l \choose i} (1-q)^i [n]^i q^{n(m+l-i)}$$
for any $l \geq 0$, 
the right-hand side of \eqref{eq5} turns into  
\begin{align*}
& u_1 u_2 \{ 1-(1-q) t u_2 \} \sum_{i, j, l, m \geq 0} {i+j \choose i} {l+m \choose l} t^{i+l} (1-t)^m
u_1^{j} u_2^{i+l+m} \\
&\quad \times  \sum_{n=1}^{\infty} \sum_{p=0}^{i+l} {i+ l \choose p} \f{(1-q)^p [n]^p q^{n(i+ j+m+l+1-p)}}{[n]^{i+j+l+m+2}}\\
& = u_1 u_2 \{ 1-(1-q) t u_2 \} \sum_{i, j, l, m \geq 0} {i+j \choose i} {l+m \choose l} t^{i+l} (1-t)^m u_1^{j} u_2^{i+l+m} \\
& \quad \times \sum_{p=0}^{i+l} {i+ l \choose p} (1-q)^p \q(2+i+j+l+m-p)\\
& = \sum_{{i, j, l, m \geq 0 \atop 0 \leq p \leq i+l }} {i+j \choose i} {l+m \choose l}{i+ l \choose p} t^{i+l} (1-t)^m
(1-q)^p \q(2+i+j+l+m-p) \\
& \quad \times u_1^{j+1} u_2^{i+l+m+1} \\
 & -  \sum_{{i, j, l, m \geq 0 \atop 0 \leq p \leq i+l }} {i+j \choose i} {l+m \choose l}{i+ l \choose p} t^{i+l+1} (1-t)^m
(1-q)^{p+1} \q(2+i+j+l+m-p)\\
& \quad \times u_1^{j+1} u_2^{i+l+m+2}. 
\end{align*}
Comparing the coefficients of $u_1^{k-n}u_2^{n}$ 
for $k, n$ with $k > n \geq 1$, we have
\begin{align*}
& \sum_{\text{wt}(\k)=k, \text{dep}(\k)=n \atop \k : \text{admissible}} \tq(\k) \\
& = 
\sum_{{i+l+m=n-1 \atop i, l, m \geq 0}  \atop 0 \leq p \leq i+l} 
 {i+k-n-1 \choose i} {l+m \choose l}{i+ l \choose p} t^{i+l} (1-t)^m
(1-q)^p \q(k-p)\\
& -\sum_{{i+l+m=n-2  \atop i, l, m \geq 0}  \atop 1 \leq p \leq i+l+1} 
 {i+k-n-1 \choose i} {l+m \choose l}{i+ l \choose p-1} t^{i+l+1} (1-t)^m
(1-q)^{p} \q(k-p)\\
& = 
\sum_{{j+m=n-1 \atop j, m \geq 0}  \atop 0 \leq p \leq j} 
\sum_{i=0}^{j}
 {i+k-n-1 \choose i} {j-i+m \choose j-i}{j \choose p} t^{j} (1-t)^m (1-q)^p \q(k-p)\\
& -\sum_{{j+m=n-1  \atop  j \geq 1, m \geq 0}  \atop 1 \leq p \leq j} 
\sum_{i=0}^{j-1}
 {i+k-n-1 \choose i} {j-i+m-1 \choose j-i-1}{j-1 \choose p-1} t^j (1-t)^m (1-q)^{p} \q(k-p). \\
\end{align*}
%
Using 
$\sum_{i=0}^j {i+k \choose i} {j-i+m \choose j-i} = {k+m+j+1 \choose j}$, 
\begin{align*}
& \sum_{\text{wt}(\k)=k, \text{dep}(\k)=n \atop \k : \text{admissible}} \tq(\k) \\
& = \left\{ \sum_{{j+m=n-1 \atop j, m \geq 0}  \atop 0 \leq p \leq j} 
 {k-n+m+j \choose j} {j \choose p}\right.
 \left. -\sum_{{j+m=n-1  \atop  j \geq 1, m \geq 0}  \atop 1 \leq p \leq j} 
 {k-n+m+j-1 \choose j-1} {j-1 \choose p-1} \right\}\\
 & \qquad \times t^j (1-t)^m (1-q)^{p} \q(k-p) \\
& = \left\{ \sum_{0 \leq j \leq n-1  \atop 0 \leq p \leq j} {k-1 \choose j} {j \choose p} 
-\sum_{1 \leq j \leq n-1  \atop 1 \leq p \leq j} {k-2 \choose j-1} {j-1 \choose p-1} \right\} \\
& \qquad \times t^j (1-t)^{n-1-j} (1-q)^{p} \q(k-p) \\
& = \left\{ \sum_{0 \leq j \leq n-1  \atop 0 \leq p \leq j} {k-1 \choose j} {j \choose p} 
-\sum_{1 \leq j \leq n-1  \atop 1 \leq p \leq j} \f{p}{k-1}{k-1 \choose j} {j \choose p} \right\} \\
& \qquad \times t^j (1-t)^{n-1-j} (1-q)^{p} \q(k-p) \\
& = \f{1}{k-1} \sum_{0 \leq p \leq j \leq n-1} (k-1-p) {k-1 \choose j} {j \choose p} t^j (1-t)^{n-1-j} (1-q)^{p} \q(k-p).  
\end{align*}
Hence we obtain the theorem. 
\end{proof}
\subsubsection{Sum formula for $t$-$q$MZVs II}
As another application of Corollary \ref{cor}, we establish the following sum formula (of full height) for $t$-$q$MZVs. 
\begin{thm}\label{full}
We have 
\begin{align*}
& \sum_{k, l \geq 0} \left( \sum_{\k \in I_0(k, l, l)} \tq(\k) \right) u_1^{k-2l} u_3^l \\
& = \exp \left\{ \sum_{n=2}^{\infty} \q(n) \sum_{m=0}^{\infty} \f{(q-1)^m}{m+n} (w_1^{m+n} + w_2^{m+n} -z_1^{m+n} - z_2^{m+n}) \right\} -1, 
\end{align*}
where $z_1, z_2, w_1, w_2$ are determined by  
$$
\left\{ \begin{array}{l}
z_1 + z_2 = u_1 -(1-t)(1-q) u_3, \\ z_1 z_2 = (1-t) u_3, 
\end{array} \right. 
\left\{ \begin{array}{l}
w_1 + w_2 = u_1 +t(1-q) u_3, \\ w_1 w_2 = -tu_3. 
\end{array} \right. $$
\end{thm}
\begin{rmk}
Theorem \ref{full} implies that the sum of $t$-$q$MZVs $\tq(\k)$ with wt($\k$)$=k$ and dep($\k$)$=$ht($\k$)$=l$ can be written as a polynomial in Riemann $q$-zeta values. 
If $t=1$, this formula is reduced to a formula for $q$MZSVs with full height, which is proved in Takeyama\cite{Take}. 
Taking the limit as $q \to 1$, we obtain Li-Qin\cite[Corollary 3.6]{LQ}. 
\end{rmk} 
\begin{proof}[Proof of Theorem \ref{full}]
Applying \eqref{eq2} to Corollary \ref{cor} and  
substituting $u_2=0$, 
we have 
$$\Psi^t_0 = \f{q u_3}{1-q u_1 -t q u_3} \f{1-\f{b_1}{q}}{1-\f{b_1b_2}{q a_1a_2}} 
{}_{3} \phi_{2} 
 \left[ \begin{array}{c} q, \f{b_2}{a_1}, \f{b_2}{a_2} \\ b_2, q^2 \end{array} ; q, \f{b_1}{q} \right].$$ 
By the definition of ${}_3\phi_2$, 
\begin{align*}
\Psi^t_0 & = \f{q u_3}{1-q u_1 -t q u_3} \f{1-\f{b_1}{q}}{1-\f{b_1b_2}{q a_1a_2}} 
\sum_{n=0}^{\infty} \f{ \left( \f{b_2}{a_1}; q\right)_n \left( \f{b_2}{a_2}; q \right)_n}{(q^2; q)_n (b_2; q)_n} \left( \f{b_1}{q} \right)^n\\
& = \f{q u_3}{1-q u_1 -t q u_3} \f{1-\f{b_1}{q}}{1-\f{b_1b_2}{q a_1a_2}} 
\sum_{n=1}^{\infty} \f{(1-q) \left( 1- \f{b_2}{q}\right)}{\left( 1-\f{b_2}{q a_1} \right) \left( 1-\f{b_2}{q a_2} \right)}\f{ \left( \f{b_2}{q a_1}; q\right)_n \left( \f{b_2}{q a_2}; q \right)_n}{(q; q)_n \left( \f{b_2}{q}; q \right)_n} \left( \f{b_1}{q} \right)^{n-1}\\
& = \f{q u_3}{1-q u_1 -t q u_3} \f{\left( 1-\f{b_1}{q} \right) \left( 1-\f{b_2}{q} \right) (1-q)}{\left( 1-\f{b_1b_2}{q a_1a_2} \right) \left( 1-\f{b_2}{q a_1} \right) \left(1-\f{b_2}{q a_2} \right)} \f{q}{b_1} \left\{ {}_2 \phi_1 \left[ \begin{array}{c} \f{b_2}{q a_1}, \f{b_2}{q a_2} \\ \f{b_2}{q} \end{array} ; q, \f{b_1}{q} \right] -1 \right\}.  
\end{align*} 
Because of 
$$a_1 a_2 = q^2\{ 1+(1-q)u_1 \},\ b_1b_2= q^4 \{ 1+(1-q)u_1\} = q^2a_1a_2, $$
$$\f{1}{a_1} + \f{1}{a_2} = \f{1}{q} \f{1+(1-q) u_1 - (1-t) (1-q)^2 u_3}{1+(1-q)u_1}, $$  and 
$$\f{1}{b_1} + \f{1}{b_2} = \f{1}{q^2} \f{2+(1-q) u_1 +t (1-q)^2 u_3}{1+(1-q)u_1}, $$
the quantity  
$$\f{\left( 1-\f{b_1}{q} \right) \left( 1-\f{b_2}{q} \right) (1-q)}{\left( 1-\f{b_1b_2}{q a_1a_2} \right) \left( 1-\f{b_2}{q a_1} \right) \left(1-\f{b_2}{q a_2} \right) b_1}$$
is calculated as 
\begin{align*}
\f{1-\f{1}{q}(b_1+b_2) + \f{1}{q^2}b_1b_2}{b_1 - \f{b_1b_2}{q}\left(\f{1}{a_1} + \f{1}{a_2} \right) +b_2}
& = \f{\f{1}{b_1b_2} - \f{1}{q} \left( \f{1}{b_1} + \f{1}{b_2} \right) + \f{1}{q^2} }{\left( \f{1}{b_1} + \f{1}{b_2} \right) - \f{1}{q} \left( \f{1}{a_1} + \f{1}{a_2} \right)}\\
& = \f{\f{(1-q)^2 (1-q u_1 - t q u_3)}{q^4 (1 + (1-q) u_1)}}{\f{(1-q)^2 u_3}{q^2 (1+ (1-q) u_1)}} 
= \f{1-q u_1 - t q u_3}{q^2 u_3}. 
\end{align*}
Therefore we have 
\begin{align}\label{eq6}
\Psi^t_0 = {}_2 \phi_1 \left[ \begin{array}{c} \f{b_2}{q a_1}, \f{b_2}{q a_2} \\ \f{b_2}{q} \end{array} ; q, \f{b_1}{q} \right] -1. 
\end{align}
Here using Heine's summation formula \eqref{Heine} 
and 
$$1-\f{q^n}{[n]}x = \f{1-q^n \{1+(1-q) x \}}{1-q^n}, $$
the right-hand side of \eqref{eq6} is calculated as
\begin{align}\label{eq7}
\Psi^t_0  = \f{\left( a_1 ; q \right)_{\infty} \left( a_2 ; q \right)_{\infty}}{\left( \f{b_1}{q} ; q \right)_{\infty} \left( \f{b_2}{q} ; q \right)_{\infty}} -1 
& = \prod_{n=1}^{\infty} \f{(1-a_1 q^{n-1}) (1-a_2 q^{n-1})}{\left( 1-\f{b_1}{q} q^{n-1} \right) \left( 1-\f{b_2}{q}q^{n-1} \right)}-1 \nonumber \\
& = \prod_{n=1}^{\infty} \f{\left(1-\f{q^{n}}{[n]} z_1 \right) \left(1-\f{q^{n}}{[n]} z_2 \right)}{\left(1-\f{q^{n}}{[n]} w_1 \right)\left(1-\f{q^{n}}{[n]} w_2 \right)}-1, 
\end{align}
where $$z_i = -\f{\alpha_i}{1 + (1-q)\alpha_i},\ w_i = -\f{\beta_i}{1 + (1-q)\beta_i}\quad (i=1, 2)$$ i.e. 
$$
\left\{ \begin{array}{l}
z_1 + z_2 = u_1 -(1-t)(1-q) u_3, \\ z_1 z_2 = (1-t) u_3, 
\end{array} \right. 
\left\{ \begin{array}{l}
w_1 + w_2 = u_1 +t(1-q) u_3, \\ w_1 w_2 = -tu_3. 
\end{array} \right. $$
Using the formula \eqref{log}, 
the right-hand side of \eqref{eq7} turns into   
\begin{align*}
\Psi^t_0  & = \exp \left\{ \f{1}{q-1} \log \f{\{1+ (1-q)z_1 \} \{1+ (1-q)z_2 \}}{\{1+ (1-q)w_1 \} \{1+ (1-q)w_2 \}}\sum_{n=1}^{\infty}\f{q^n}{[n]} \right. \\
& \quad \left. - \sum_{n=2}^{\infty} \q(n) \sum_{m=0}^{\infty} \f{(q-1)^m}{m+n} (z_1^{m+n} + z_2^{m+n} -w_1^{m+n} - w_2^{m+n} \right\} -1 \\
& = \exp \left\{ \sum_{n=2}^{\infty} \q(n) \sum_{m=0}^{\infty} \f{(q-1)^m}{m+n} (w_1^{m+n} + w_2^{m+n} -z_1^{m+n} - z_2^{m+n}) \right\} -1.  
\end{align*}
Hence we obtain the theorem. 
\end{proof}
\subsection{The case of $t = 0$}
When $t=0$ our main theorem reduces to the result in Li\cite{Li} as follows. 
Note that, in the case of $t=0$,  
$$\left\{ \begin{array}{ll}
\alpha_1 + \cs \alpha_{r+1} = x_2 - x_1, & \\
\d\sum_{1 \leq i_1 < \cs < i_j \leq r+1} \alpha_{i_1} \cs \alpha_{i_j} = x_{j+1} - x_1 x_j & j = 2, \ls, r+1, 
\end{array}\right. 
$$
and
$$\left\{ \begin{array}{ll}
\beta_1 + \cs \beta_{r+1} = - x_1, & \\
\d\sum_{1 \leq i_1 < \cs < i_j \leq r+1} \beta_{i_1} \cs \beta_{i_j} = 0, & j = 2, \ls, r+1.  
\end{array}\right. 
$$
and hence we may assume 
$$\beta_1 = \cs =\beta_r = 0, \ \beta_{r+1} = -x_1. $$
\begin{thm}\label{qLi}(Li\cite{Li})
We have 
\begin{align*} 
\Psi^0_0 & = 
\f{\{ 1+(1-q) u_1 \}^2}{u_{r+2} - u _1 u_{r+1}}
 \left\{ \sum_{j=0}^{r-1} A_j^\ast B_j^\ast {}_{r+1} \phi_{r} \left[ \begin{array}{c} a_1^\ast q^j, \ls, a_{r+1}^\ast q^j \\ \underbrace{q^{j+1}, \ls, q^{j+1}}_{r-1}, b^\ast q^j \end{array} ; q, \f{b^\ast}{a_1^\ast \cs a_{r+1}^\ast} \right] - A_0^\ast \right\},  
\end{align*}
where 
$$a_i^\ast = \f{a_i}{q} = \f{1}{1+(1-q)\alpha_i},\ b^\ast = \f{b_{r+1}}{q} = q \{ 1+(1-q)u_1 \}, $$
$$A^\ast_j = \sum_{m=j}^{r-1} c_m S_q(m, j) + \f{u_1 u_2}{\{ 1+(1-q)u_1\}^2} S_q(r-1, j), $$
$$B_j^\ast = \f{(a_1^\ast, \ls, a_{r+1}^\ast ; q)_j}{(\underbrace{q, \ls, q}_{r-1}, b^\ast ; q)_j} \left( \f{b^\ast}{(1-q)a_1^\ast \cs a_{r+1}^\ast} \right)^j,  $$
for $1 \leq i \leq r+1,\ 0 \leq j \leq r-1$ 
and $\alpha_i$'s are as above. 
\end{thm}
\begin{proof}
Since $\beta_1 = \cs =\beta_r = 0, \ \beta_{r+1} = -x_1$, 
we see that 
$$b_1 = \cs = b_r = q^2,\ 
b_{r+1} = \f{q^2}{1-(1-q)x_1} = q^2\{1+(1-q)u_1 \}. $$
Since
$$S_q(m+1, j+1) = q^j S_q(m, j) + [j+1] S_q(m, j+1)\ \ (0 \leq j \leq m), $$
$A_j$'s in the main theorem are written as 
$$A_j = q^j A_j^\ast + [j+1]A_{j+1}^\ast \quad (0 \leq j \leq r-1,\ A_r^\ast = 0). $$
We also find 
\begin{align*}
\widetilde{B_j} & = \f{q (1-q )^{r+1} (1-b^\ast) a_1^\ast \cs a_{r+1}^\ast}{(1-q^{j+1}) b^\ast (1-a_1^\ast) \cs (1-a_{r+1}^\ast)} B_{j+1}^\ast\\
& = \f{q (1-q )^{r} (1-b^\ast)}{(1-q^{j+1}) (1-a_1^\ast) \cs (1-a_{r+1}^\ast)} \f{ (1-a_1^\ast q^j) \cs (1-a_{r+1}^\ast q^j)}{(1-q^{j+1})^{r-1} (1-b^\ast q^j)} B_{j}^\ast. 
\end{align*}
Applying these to the main theorem, we have 
\begin{align*}
 \Psi^0_0 & = \f{\{ 1+(1-q) u_1 \}^2}{1-q u_1}
 \\
&\times \left\{ A_0^\ast q \ {}_{r+2} \phi_{r+1} \left[ \begin{array}{c} q, a_1^\ast q, \ls, a_{r+1}^\ast q \\ \underbrace{q^2, \ls, q^2}_r, b^\ast q \end{array} ; q, \f{b^\ast}{a_1^\ast \cs a_{r+1}^\ast} \right] \right. 
\\
& \quad + \sum_{j=1}^{r-1} q^j A_j^\ast \f{q (1-q )^{r} (1-b^\ast)}{(1-q^{j+1}) (1-a_1^\ast) \cs (1-a_{r+1}^\ast)} \f{ (1-a_1^\ast q^j) \cs (1-a_{r+1}^\ast q^j)}{(1-q^{j+1})^{r-1} (1-b^\ast q^j)} B_{j}^\ast \nonumber\\
& \qquad \times {}_{r+2} \phi_{r+1} \left[ \begin{array}{c} q^{j+1}, a_1^\ast q^{j+1}, \ls, a_{r+1}^\ast q^{j+1} \\ \underbrace{q^{j+2}, \ls, q^{j+2}}_r, b^\ast q^{j+1} \end{array} ; q, \f{b^\ast}{ a_1^\ast \cs a_{r+1}^\ast} \right] 
\\
& \quad + \sum_{j=1}^{r-1} [j] A_j^\ast \f{q (1-q )^{r+1} (1-b^\ast) a_1^\ast \cs a_{r+1}^\ast}{(1-q^{j}) b^\ast (1-a_1^\ast) \cs (1-a_{r+1}^\ast)} B_{j}^\ast 
\\
& \qquad \left. \times {}_{r+2} \phi_{r+1} \left[ \begin{array}{c} q^{j}, a_1^\ast q^{j}, \ls, a_{r+1}^\ast q^{j} \\ \underbrace{q^{j+1}, \ls, q^{j+1}}_r, b^\ast q^{j} \end{array} ; q, \f{b^\ast}{ a_1^\ast \cs a_{r+1}^\ast} \right] \right\}. 
\end{align*}
Let $\sum_1, \sum_2, \sum_3$ be the first, the second, and the third terms in the bracket of the right-hand side, respectively. 
Then we have
\begin{align*}
{\sum}_1 
& = A_0^\ast q \sum_{n=0}^{\infty}  \f{(a_1^\ast q, \ls, a_{r+1}^\ast q ; q)_n}{(\underbrace{q^2, \ls, q^2}_{r}, b^\ast q ; q)_n} \left( \f{b^\ast}{a_1^\ast \cs a_{r+1}^\ast} \right)^n \\
& = A_0^\ast q \f{(1-q)^r (1-b^\ast)}{(1-a_1^\ast) \cs (1-a_{r+1}^\ast)} \f{a_1^\ast \cs a_{r+1}^\ast}{b^\ast} 
\sum_{n=1}^{\infty}  \f{(a_1^\ast, \ls, a_{r+1}^\ast ; q)_n}{(\underbrace{q, \ls, q}_{r}, b^\ast ; q)_n} \left( \f{b^\ast}{a_1^\ast \cs a_{r+1}^\ast} \right)^n \\
& = A_0^\ast q \f{(1-q)^r (1-b^\ast) a_1^\ast \cs a_{r+1}^\ast }{b^\ast(1-a_1^\ast) \cs (1-a_{r+1}^\ast)} 
\left\{ {}_{r+1} \phi_{r} \left[ \begin{array}{c} a_1^\ast, \ls, a_{r+1}^\ast \\ \underbrace{q, \ls, q}_{r-1}, b^\ast \end{array} ; q, \f{b^\ast}{a_1^\ast \cs a_{r+1}^\ast} \right] -1 \right\},  
\end{align*}
\begin{align*}
{\sum}_2 & = \f{q(1-q)^r (1-b^\ast) a_1^\ast \cs a_{r+1}^\ast }{b^\ast (1-a_1^\ast) \cs (1-a_{r+1}^\ast)} \\
& \quad \times \sum_{j=1}^{r-1} A_j^\ast B_j^\ast \sum_{n=1}^{\infty} 
\f{q^{j} (q^{j+1} ; q)_{n-1} (a_1^\ast q^j, \ls, a_{r+1}^\ast q^j; q)_n}{(q ; q)_{n-1} (\underbrace{q^{j+1}, \ls, q^{j+1}}_r, b^\ast q^j ; q)_n} 
\left( \f{b^\ast}{a_1^\ast \cs a_{r+1}^\ast} \right)^n
\end{align*}
(using $(1-a^\ast q^j)(a^\ast q^{j+1} ; q)_n = (a^\ast q^j ; q)_{n+1}$), and 
\begin{align*}
{\sum}_3 
& = \f{q (1-q)^r (1-b^\ast) a_1^\ast \cs a_{r+1}^\ast }{b^\ast (1-a_1^\ast) \cs (1-a_{r+1}^\ast)}\\
& \quad \times \sum_{j=1}^{r-1} A_j^\ast B_j^\ast \left\{ 1+ \sum_{n=1}^{\infty} 
\f{(q^{j} ; q)_{n} (a_1^\ast q^j, \ls, a_{r+1}^\ast q^j; q)_n}{(q ; q)_{n} (\underbrace{q^{j+1}, \ls, q^{j+1}}_r, b^\ast q^j ; q)_n} 
\left( \f{b^\ast}{a_1^\ast \cs a_{r+1}^\ast} \right)^n \right\}. 
\end{align*}
Since 
$$\f{q^j (q^{j+1}; q)_{n-1}}{(q ;  q)_{n-1}} + \f{(q^j ; q)_n}{(q ; q)_n} = \f{(q^{j+1} ; q)_n}{(q; q)_n}, $$
we have 
\begin{align*}
& {\sum}_2 + {\sum}_3 \\
& =
\f{q (1-q)^r (1-b^\ast) a_1^\ast \cs a_{r+1}^\ast }{b^\ast (1-a_1^\ast) \cs (1-a_{r+1}^\ast)}\\
& \quad \times \sum_{j=1}^{r-1} A_j^\ast B_j^\ast \left\{ 1+ \sum_{n=1}^{\infty} 
\f{(a_1^\ast q^j, \ls, a_{r+1}^\ast q^j; q)_n}{(q ; q)_{n} (\underbrace{q^{j+1}, \ls, q^{j+1}}_{r-1}, b^\ast q^j ; q)_n} 
\left( \f{b^\ast}{a_1^\ast \cs a_{r+1}^\ast} \right)^n \right\} \\
& = \f{q (1-q)^r (1-b^\ast) a_1^\ast \cs a_{r+1}^\ast }{b^\ast (1-a_1^\ast) \cs (1-a_{r+1}^\ast)}
\sum_{j=1}^{r-1} A_j^\ast B_j^\ast {}_{r+1} \phi_{r} \left[ \begin{array}{c} a_1^\ast q^j, \ls, a_{r+1}^\ast q^j \\ \underbrace{q^{j+1}, \ls, q^{j+1}}_{r-1}, b^\ast q^j \end{array} ; q, \f{b^\ast}{a_1^\ast \cs a_{r+1}^\ast} \right]. 
\end{align*} 
Therefore we obtain
\begin{align*}
\Psi^t_0 & = \f{\{ 1+ (1-q) u_1 \}^2}{1-q u_1} \f{q (1-q)^r (1-b^\ast) a_1^\ast \cs a_{r+1}^\ast }{b^\ast (1-a_1^\ast) \cs (1-a_{r+1}^\ast)} \\
& \quad \times \left\{ \sum_{j=0}^{r-1} A_j^\ast B_j^\ast {}_{r+1} \phi_{r} \left[ \begin{array}{c} a_1^\ast q^j, \ls, a_{r+1}^\ast q^j \\ \underbrace{q^{j+1}, \ls, q^{j+1}}_{r-1}, b^\ast q^j \end{array} ; q, \f{b^\ast}{a_1^\ast \cs a_{r+1}^\ast} \right] - A_0^\ast \right\}. 
\end{align*}
We find 
$$\f{a_1^{\ast}}{1-a_1^{\ast}} \cs \f{a_{r+1}^{\ast}}{1-a_{r+1}^{\ast}} = \f{1}{(1-q)^{r+1} (x_{r+2} - x_1 x_{r+1})}, $$
$$\f{q(1-q)^r}{1-q u_1}\f{1-b^\ast}{b^\ast} = \f{(1-q)^{r+1}}{1+(1-q)u_1}. $$
Also by \eqref{x1} and \eqref{xj}, we have 
$$x_{r+2} - x_1 x_{r+1} = \f{u_{r+2} - u_1 u_{r+1}}{1+(1-q)u_1}. $$
Hence we conclude the theorem. 
\end{proof}

\section{$t$-$q$MPLs}\label{tqMPL}
For any multi-index $\k = (k_1, k_2, \ls, k_l) \in \N^l$ and a parameter $q$, $q$-analogue of multiple polylogarithms (of non-star and star types, collectively abbreviated as $q$MPLs) are defined by 
\begin{align*}
Li_{\k; q}(z) & =\sum_{m_1 > m_2 > \cs > m_l \geq 1} \f{z^{m_1}}{{[m_1]}^{k_1}{[m_2]}^{k_2} \cs {[m_l]}^{k_l}} \ (\in \Q[[q, z]]),  \\
Li^{\star}_{\k; q}(z) & =\sum_{m_1 \geq m_2 \geq \cs \geq m_l \geq 1} \f{z^{m_1}}{{[m_1]}^{k_1}{[m_2]}^{k_2} \cs {[m_l]}^{k_l}} \ (\in \Q[[q, z]]). 
\end{align*}
These series converge if $|z| < 1$ for any $\k$. 
We also define $t$-$q$MPLs by 
\begin{align*}
Li^{t}_{\k; q}(z) = {\sum_{\p}}' Li_{\p; q}(z) t^{l - \rm{dep}(\p)} \ (\in \Q[[q, z]][t]),  
\end{align*}
where $\sum_{\p}'$ stands for the sum where 
$\p$ runs over all indices of the form  
$\p=(k_1\ \square\ \cs\ \square\ k_l)$ in which each $\square$ is filled by two candidates:
$``,"$ or $``+"$. 
\subsection{Difference formula for $t$-$q$MPLs $Li^{t}_{\k; q}(z)$}
Denote by $\Dq$ the $q$-difference operator 
$$(\Dq f)(z) =\f{f(z)-f(qz)}{(1-q)z}. $$
Then by definition $q$MPLs satisfy 
$$\Dq Li_{(k_1, k_2, \ldots, k_l); q}(z) =
\left\{ \begin{array}{ll}
\d\f{1}{z}Li_{(k_1-1, k_2, \ldots, k_l); q}(z), & k_1 \geq 2, \\
\d\f{1}{1-z} Li_{(k_2, \ldots, k_l); q}(z), & k_1=1, l \geq 2, \\
\d\f{1}{1-z}, & k_1 = l=1. 
\end{array}\right. $$

\begin{lem}\label{lem2}
For positive inegers $k_1, k_2, \ldots, k_l$, 
$t$-$q$MPLs satisfy 
$$\Dq Li^t_{(k_1, k_2, \ldots, k_l); q}(z) =
\left\{ \begin{array}{ll}
\d\f{1}{z}Li^t_{(k_1-1, k_2, \ldots, k_l); q}(z), & k_1 \geq 2, \\
\d\left( \f{t}{z} + \f{1}{1-z}\right) Li^t_{(k_2, \ldots, k_l); q}(z), & k_1=1, l \geq 2, \\
\d\f{1}{1-z}, & k_1 = l=1. 
\end{array}\right. $$
\end{lem}
\begin{proof}
If $l \geq 2$, we have the decomposition 
\begin{align*}
Li^{t}_{(k_1, k_2, \ls, k_l); q}(z) 
= \sum_{{\p=(k_1, k_2\ \square\ \cs\ \square\ k_l)} \atop {\square = ``," \text{or} ``+"}} Li_{\p; q}(z) t^{l - \rm{dep}(\p)}
+ \sum_{{\p=(k_1+ k_2\ \square\ \cs\ \square\ k_l)} \atop {\square = ``," \text{or} ``+"}} Li_{\p; q}(z) t^{l - \rm{dep}(\p)}. 
\end{align*}
Hence if $k_1 \geq 2$, we have 
\begin{align*}
 & \Dq Li^{t}_{(k_1, k_2, \ls, k_l); q}(z) \\
 & = \f{1}{z} \sum_{{\p=(k_1-1, k_2\ \square\ \cs\ \square\ k_l)} \atop {\square = ``," \text{or} ``+"}} Li_{\p; q}(z) t^{l - \rm{dep}(\p)}
+ \f{1}{z} \sum_{{\p=(k_1-1+ k_2\ \square\ \cs\ \square\ k_l)} \atop {\square = ``," \text{or} ``+"}} Li_{\p; q}(z) t^{l - \rm{dep}(\p)} \\
 &= \f{1}{z}Li^{t}_{(k_1-1, k_2, \ls, k_l); q}(z),  
 \end{align*}
and otherwise 
\begin{align*}
 & \Dq Li^{t}_{(1, k_2, \ls, k_l); q}(z) \\
 & = \f{1}{1-z} \sum_{{\p=(k_2\ \square\ \cs\ \square\ k_l)} \atop {\square = ``," \text{or} ``+"}} Li_{\p; q}(z) t^{l - \rm{dep}(\p)-1}
+ \f{1}{z} \sum_{{\p=(k_2\ \square\ \cs\ \square\ k_l)} \atop {\square = ``," \text{or} ``+"}} Li_{\p; q}(z) t^{l - \rm{dep}(\p)} \\
 &= \left( \f{t}{z} + \f{1}{1-z}\right) Li^{t}_{(k_2, \ls, k_l); q}(z). 
 \end{align*}
If $l=1$, since $Li^{t}_{(k_1); q}(z) = Li_{(k_1); q}(z)$, we have
$$\Dq Li^t_{(k_1); q}(z) =
\left\{ \begin{array}{ll}
\d\f{1}{z}Li_{(k_1-1); q}(z) = \f{1}{z}Li^t_{(k_1-1); q}(z), & k_1 \geq 2, \\
\d\f{1}{1-z}, & k_1 =1. 
\end{array}\right. $$
\end{proof}
\subsection{Difference formula for sums of $t$-$q$MPLs $G^t_j(z)$}
Let $r$ be a positive integer. 
We define sets of indices $I_j\ (-1 \leq j \leq r-1), I$ for any integers $k, l, h_1, h_2, \ldots, h_r \geq 0$ by  
\begin{align*}
I_j&(k, l, h_1, h_2, \ldots, h_r) \\
 & = \{ \k = (k_1, k_2, \ls, k_l) \in \N^l | \text{wt}(\k)=k, i\text{-ht}(\k)=h_i (i=1, \ldots, r), k_1 \geq j+2\} \\
\text{and} \\
I&(k, l, h_1, h_2, \ldots, h_r) =I_{-1}(k, l, h_1, h_2, \ldots, h_r) .
\end{align*}
($I_0$ is nothing but that we defined in $\S 2$.)
We notice that: 
\begin{itemize}
\item[(i)] $I_j(k, l, h_1, \ls, h_r) \subseteq I_{j+1}(k, l, h_1, \ls, h_r)$. 
\item[(ii)] 
If $I(k, l, h_1, h_2, \ldots, h_r) \neq \phi$, then 
$$\left\{ \begin{array}{l}
k \geq l+ \d\sum_{j=1}^{r}h_j, \\
l \geq h_1 \geq \cdots \geq h_r \geq 0. 
\end{array}\right. $$
\item[(iii)] 
If $I_j(k, l, h_1, h_2, \ldots, h_r) \neq \phi$, then 
$$\left\{ \begin{array}{l}
k \geq l+ \d\sum_{j=1}^{r}h_j, \\
l \geq h_1 \geq \cdots \geq h_r \geq 0,\\
h_{j+1} \geq 1,  
\end{array}\right. $$
for $j= 0, 1, \ldots, r-1$.  
\end{itemize}

For any integers $k, l, h_1, \ldots, h_r \geq 0$ 
and 
$-1 \leq j \leq r-1$, 
we set
\begin{align*}
G^t_j(z) = G^t_j(k, l, h_1, \ldots, h_r; z) & = \sum_{\k \in I_j(k, l, h_1, \ldots, h_r)} Li^t_{\k; q}(z), \\
G^t(z) = G^t(k, l, h_1, \ldots, h_r; z) & = G^t_{-1}(k, l, h_1, \ldots, h_r; z).  
\end{align*} 
In the definition,  the sum is regarded as $0$ 
whenever the index set is empty except for $G^t(0, 0, \ldots, 0; z) : =1$. 

Applying the $q$-difference oparator $\Dq$ to the sum $G^t_j(z)$, 
we obtain the following lemma. 
\begin{lem}\label{lem3}
\begin{itemize}
\item[(i)]  For $k \geq l+\sum_{j=1}^{r} h_j$, $l \geq h_1 \geq \cs \geq h_r \geq 1$, we have  
\begin{align*}
& \Dq G^t_{r-1}(k, l, h_1, \ls, h_r; z) \\
& = \f{1}{z} \left\{ G^t_{r-1}(k-1, l, h_1, \ls, h_r; z) \right. \\
& \quad + \left. G^t_{r-2}(k-1, l, h_1, \ls, h_{r-1}, h_r-1; z) -G^t_{r-1}(k-1, l, h_1, \ls, h_{r-1}, h_r-1; z)
\right\}. 
\end{align*}
\item[(ii)] For $ 0 \leq j \leq r-2$, $k \geq l+\sum_{j=1}^{r} h_j$, $l \geq h_1 \geq \cs \geq h_r \geq 0$, $h_{j+1} \geq 1$, we have 
\begin{align*}
& \Dq \left\{ G^t_j(k, l, h_1, \ls, h_r; z) - G^t_{j+1}(k, l, h_1, \ls, h_r; z) \right\} \\
& = \f{1}{z}\left\{ G^t_{j-1}(k-1, l, h_1, \ls, h_j, h_{j+1}-1, h_{j+2}, \ls, h_r; z) \right. \\
& \quad -\left. G^t_{j}(k-1, l, h_1, \ls, h_j, h_{j+1}-1, h_{j+2}, \ls, h_r; z)\right\}. 
\end{align*}
\item[(iii)]For $k \geq l+\sum_{j=1}^{r} h_j$, $l \geq h_1 \geq \cs \geq h_r \geq 0$, $l \geq 2$, we have 
\begin{align*}
& \Dq \left\{ G^t(k, l, h_1, \ls, h_r; z) - G^t_{0}(k, l, h_1, \ls, h_r; z) \right\} \\
& = \left( \f{t}{z} + \f{1}{1-z} \right) G^t(k-1, l-1, h_1, \ls, h_r; z). \end{align*}
\end{itemize}
\end{lem} 
\begin{proof}
\begin{itemize}
\item[(i)] If $\k=(k_1, \ls, k_l) \in I_{r-1}(k, l, h_1, \ls, h_r)$, $k_1 \geq r+1 \geq 2$. Then we have  
\begin{align*}
& \Dq G^t_{r-1}(k, l, h_1, \ls, h_r; z) \\
& = \f{1}{z} \sum_{{\k=(k_1, \ls, k_l)}\atop{ \in I_{r-1}(k, l, h_1, \ls, h_r)}}Li^t_{(k_1-1, k_2, \ls, k_l); q}(z)\\
& =\f{1}{z} \sum_{{{\k=(k_1, \ls, k_l)}\atop{ \in I_{r-1}(k, l, h_1, \ls, h_r), }}\atop{k_1 \geq r+2}}Li^t_{(k_1-1, k_2, \ls, k_l); q}(z)
+ \f{1}{z} \sum_{{{\k=(k_1, \ls, k_l)}\atop{ \in I_{r-1}(k, l, h_1, \ls, h_r), }}\atop{k_1 = r+1}}Li^t_{(k_1-1, k_2, \ls, k_l); q}(z)\\
& = \f{1}{z} G^t_{r-1}(k-1, l, h_1, \ls, h_r; z) +\f{1}{z}\sum_{{{\k=(k_1, \ls, k_l)}\atop{ \in I_{r-2}(k-1, l, h_1, \ls, h_{r-1}, h_r-1), }}\atop{k_1 = r}}Li^t_{\k; q}(z)\\
& = \f{1}{z} G^t_{r-1}(k-1, l, h_1, \ls, h_r; z) \\
& \quad +\f{1}{z}\sum_{{{\k=(k_1, \ls, k_l)}\atop{ \in I_{r-2}(k-1, l, h_1, \ls, h_{r-1}, h_r-1), }}\atop{k_1 \geq r}}Li^t_{\k; q}(z)
-\f{1}{z}\sum_{{{\k=(k_1, \ls, k_l)}\atop{ \in I_{r-2}(k-1, l, h_1, \ls, h_{r-1}, h_r-1), }}\atop{k_1 \geq r + 1}}Li^t_{\k; q}(z)\\
& = \f{1}{z} G^t_{r-1}(k-1, l, h_1, \ls, h_r; z) \\
& \quad + \f{1}{z} G^t_{r-2}(k-1, l, h_1, \ls, h_{r-1}, h_r-1; z) - \f{1}{z} G^t_{r-1}(k-1, l, h_1, \ls, h_{r-1}, h_r-1; z). 
\end{align*}
\item[(ii)] For $ 0 \leq j \leq r-2$, 
if $\k=(k_1, \ls, k_l) \in I_j(k, l, h_1, \ls, h_r)
$, $k_1 \geq j+2 \geq 2$. Then we have
\begin{align*}
& \Dq \left\{ G^t_j(k, l, h_1, \ls, h_r; z) - G^t_{j+1}(k, l, h_1, \ls, h_r; z) \right\} \\
& = \Dq \sum_{{{\k=(k_1, \ls, k_l)}\atop{ \in I_{j}(k, l, h_1, \ls, h_r), }}\atop{k_1 = j+2}}Li^t_{\k; q}(z)\\
& = \f{1}{z} \sum_{{{\k=(k_1, \ls, k_l)}\atop{ \in I_{j-1}(k-1, l, h_1, \ls,h_j,  h_{j+1}-1, h_{j+2}, \ls, h_r), }}\atop{k_1 = j+1}}Li^t_{\k; q}(z)\\
& = \f{1}{z}\left\{ G^t_{j-1}(k-1, l, h_1, \ls, h_j, h_{j+1}-1, h_{j+2}, \ls, h_r; z) \right. \\
& \quad -\left. G^t_{j}(k-1, l, h_1, \ls, h_j, h_{j+1}-1, h_{j+2}, \ls, h_r; z)\right\}. 
\end{align*}
\item[(iii)] If $l \geq 2$, we have
\begin{align*}
& \Dq \left\{ G^t(k, l, h_1, \ls, h_r; z) - G^t_{0}(k, l, h_1, \ls, h_r; z) \right\} \\
&= \Dq \sum_{{{\k=(k_1, \ls, k_l)}\atop{ \in I(k, l, h_1, \ls, h_r), }}\atop{k_1 = 1}}Li^t_{\k; q}(z) 
= \left( \f{t}{z} +\f{1}{1-z} \right)\sum_{{\k=(k_2, \ls, k_l)}\atop{ \in I(k-1, l-1, h_1, \ls, h_r), }}Li^t_{\k; q}(z) \\
& = \left( \f{t}{z} + \f{1}{1-z} \right) G^t(k-1, l-1, h_1, \ls, h_r; z). \end{align*}
\end{itemize}
\end{proof}
\subsection{Difference formula for generating functions $\Phi^t_j(z)$}
Let $x_1, \ls, x_{r+2}$ be variables. 
We define generating functions for $G^t_j(z)$ ($-1 \leq j \leq r-1$) 
as follows: 
\begin{align*}
\Phi^t_j& =\Phi^t_j(z) = \Phi^t_j(x_1, \ls, x_{r+2}; z) \\
& = \sum_{k, l, h_1, \ls, h_r \geq 0} G^t_j(k, l, h_1, \ls, h_r; z) x_1^{k-l-\sum_{i=1}^{r}h_i} x_2^{n-h_1} x_3^{h_1-h_2} \cs x_{r+1}^{h_{r-1}-h_r} x_{r+2}^{h_r},\\
\Phi^t& = \Phi^t_{-1}. 
\end{align*}

Applying the $q$-difference oparator to the generating functions $\Phi^t_j, \Phi^t$, 
we obtain the following proposition. 
\begin{prop}\label{prop4}
Let $r$ be a positive integer.  We have
$$\left\{ \begin{array}{l}
\Dq \Phi^t_{r-1} = \d \f{x_1}{z} \Phi^t_{r-1} + \f{x_{r+2}}{z x_{r+1}}(\Phi^t_{r-2} -\Phi^t_{r-1}),  \\
\Dq (\Phi^t_{j} - \Phi^t_{j+1}) =\d  \f{x_{j+3}}{z x_{j+2}} (\Phi^t_{j-1} -\Phi^t_{j}), \qquad j=1, 2, \ls, r-2, \\
\Dq (\Phi^t_{0} - \Phi^t_{1}) = \d \f{x_{3}}{z x_{2}} (\Phi^t -\Phi^t_{0}-1),  \\
\Dq (\Phi^t - \Phi^t_{0}) = \d \f{x_2}{1-z} + \left( \f{t}{z} + \f{1}{1-z}\right) x_2 (\Phi^t-1).  
\end{array}
\right. $$
\end{prop}
\begin{proof}
By the definition of $\Phi^t_{r-1}$, we have
$$\Dq\Phi^t_{r-1}= \sum_{{k \geq l+\sum_{j=1}^{r}h_j} \atop {l \geq h_1 \geq \cs \geq h_r \geq 1}} \Dq G^t_{r-1}(k, l, h_1, \ls, h_r; z) x_1^{k-l-\sum_{j=1}^{r}h_j} x_2^{l-h_1} x_3^{h_1-h_2} \cs x_{r+1}^{h_{r-1}-h_r} x_{r+2}^{h_r}.$$
Applying the equation (i) of Lemma \ref{lem3} to the right-hand side,  
we have
\begin{align*}
& \Dq\Phi^t_{r-1}\\
 & = \f{1}{z} \sum_{{k \geq l+\sum_{j=1}^{r}h_j} \atop {l \geq h_1 \geq \cs \geq h_r \geq 1} } 
 G^t_{r-1} (k-1, l, h_1, \ls, h_r ; z) x_1^{k-l-\sum_{j=1}^{r}h_j} x_2^{l-h_1} x_3^{h_1-h_2} \cs x_{r+1}^{h_{r-1} - h_{r}} x_{r+2}^{h_r}\\
 &+ \f{1}{z} \sum_{{k \geq l+\sum_{j=1}^{r}h_j} \atop {l \geq h_1 \geq \cs \geq h_r \geq 1} }
 \left\{ G^t_{r-2}(k-1, l, h_1, \ls, h_{r-1}, h_r-1 ; z) \right. \\
 & \qquad \left. -G^t_{r-1}(k-1, l, h_1, \ls, h_{r-1}, h_r-1 ; z) \right\}
 x_1^{k-l-\sum_{j=1}^{r}h_j} x_2^{l-h_1} x_3^{h_1-h_2} \cs x_{r+1}^{h_{r-1} - h_{r}} x_{r+2}^{h_r}\\
 & = \f{x_1}{z} \Phi^t_{r-1} 
 + \f{1}{z} \left\{ \sum_{{k-1 \geq l+\sum_{j=1}^{r}h_j-1} \atop {l \geq h_1 \geq \cs \geq h_{r-1} \geq h_r -1 \geq 0} }
 - \sum_{{k-1 \geq l+\sum_{j=1}^{r}h_j-1} \atop {l \geq h_1 \geq \cs \geq h_{r-1}=h_r-1 \geq 0} } \right\} \\
& \times \left\{ G^t_{r-2}(k-1, l, h_1, \ls, h_{r-1}, h_r-1 ; z) \right. \\
&  \qquad \left.- G^t_{r-1}(k-1, l, h_1, \ls, h_{r-1}, h_r-1 ; z) \right\}
x_1^{k-l-\sum_{j=1}^{r}h_j} x_2^{l-h_1} x_3^{h_1-h_2} \cs x_{r+1}^{h_{r-1} - h_{r}} x_{r+2}^{h_r}\\
& = \f{x_1}{z} \Phi^t_{r-1} + \f{x_{r+2}}{z x_{r+1}}(\Phi^t_{r-2} -\Phi^t_{r-1}). 
\end{align*}
Note that the second sum is $0$ since no indices have the property $h_{r-1}=h_r-1$. 

We proceed to show the second equation. 
Since $\Phi^t_j=\Phi^t_{j+1}$ when $h_{j+1} = h_{j+2}$, we see that
\begin{align*}
& \Dq\left( \Phi^t_{j} - \Phi^t_{j+1} \right)\\
&  = 
\sum_{{{k \geq l+\sum_{i=1}^{r}h_i} \atop {l \geq h_1 \geq \cs \geq h_{j+1} > h_{j+2} \geq \cs \geq h_r \geq 0}} \atop ({h_{j+1} \geq 1) }}
\Dq\left\{ G^t_j(k, l, h_1, \ls, h_r ; z) - G^t_{j+1}(k, l, h_1, \ls, h_r ; z)\right\}\\
& \qquad \times 
x_1^{k-l-\sum_{i=1}^{r}h_i} x_2^{l-h_1} x_3^{h_1-h_2} \cs x_{r+1}^{h_{r-1} - h_{r}} x_{r+2}^{h_r}. 
\end{align*}
By applying Lemma \ref{lem3} (ii), the right-hand side turns into 
\begin{align*}
& \f{1}{z} \sum_{{{k \geq l+\sum_{i=1}^{r}h_i} \atop {l \geq h_1 \geq \cs \geq h_{j+1} > h_{j+2} \geq \cs \geq h_r \geq 0}} \atop ({h_{j+1} \geq 1) }}
\left\{ G^t_{j-1}(k-1, l, h_1, \ls, h_j, h_{j+1}-1, h_{j+2},\ls, h_r ; z)\right. \\
& \quad \left. - G^t_{j}(k-1, l, h_1, \ls, h_j, h_{j+1}-1, h_{j+2}, \ls, h_r ; z)\right\}\\
& \qquad \times 
x_1^{k-l-\sum_{i=1}^{r}h_i} x_2^{l-h_1} x_3^{h_1-h_2} \cs x_{r+1}^{h_{r-1} - h_{r}} x_{r+2}^{h_r}\\
& = \f{1}{z}\f{x_{j+3}}{x_{j+2}} 
\sum_{{{k \geq l+\sum_{i=1}^{r}h_i} \atop {l \geq h_1 \geq \cs \geq h_j > h_{j+1} \geq h_{j+2} \geq \cs \geq h_r \geq 0}} \atop ({h_{j} \geq 1) }}
\left\{ G^t_{j-1}(k, l, h_1, \ls, h_j, h_{j+1}, h_{j+2},\ls, h_r ; z)\right. \\
& \quad \left. - G^t_{j}(k, l, h_1, \ls, h_j, h_{j+1}, h_{j+2}, \ls, h_r ; z)\right\}\\
& \qquad \times 
x_1^{k-l-\sum_{i=1}^{r}h_i} x_2^{l-h_1} x_3^{h_1-h_2} \cs x_{r+1}^{h_{r-1} - h_{r}} x_{r+2}^{h_r}\\
& = \f{x_{j+3}}{z x_{j+2}} (\Phi^t_{j-1} -\Phi^t_{j}). 
\end{align*}

For the third one, note that 
$\Phi^t_0 = \Phi^t_1$ when $h_1=h_2$. 
Then,  by using Lemma \ref{lem3} (ii), we have
\begin{align*}
& \Dq(\Phi^t_{0} - \Phi^t_{1}) \\
& = \f{1}{z} \sum_{{{k \geq l+\sum_{j=1}^{r}h_j} \atop {l \geq h_1 > h_2 \geq \cs \geq h_r \geq 0}} \atop ({h_{1} \geq 1 })}
\left\{ G^t(k-1, l, h_1-1, h_2, \ls, h_r ; z)\right. \\
& \quad \left. - G^t_{0}(k-1, l, h_1-1, h_2, \ls, h_r ; z)\right\}
x_1^{k-l-\sum_{j=1}^{r}h_j} x_2^{l-h_1} x_3^{h_1-h_2} \cs x_{r+1}^{h_{r-1} - h_{r}} x_{r+2}^{h_r} \\
& = \f{1}{z} \f{x_3}{x_2}
\left\{ \sum_{{{k \geq l+\sum_{j=1}^{r}h_j} \atop {l \geq h_1 \geq \cs \geq h_r \geq 0}} \atop{l \geq 1 }}
- \sum_{{{k \geq l+\sum_{j=1}^{r}h_j} \atop {l = h_1 \geq \cs \geq h_r \geq 0}} \atop{l \geq 1 }}
\right\}
\left\{ G^t(k, l, h_1, h_2, \ls, h_r ; z)\right. \\
& \quad \left. - G^t_{0}(k, l, h_1, h_2, \ls, h_r ; z)\right\}
x_1^{k-l-\sum_{j=1}^{r}h_j} x_2^{l-h_1} x_3^{h_1-h_2} \cs x_{r+1}^{h_{r-1} - h_{r}} x_{r+2}^{h_r}. 
\end{align*}
The second sum is $0$ 
because 
\begin{align}\label{II0}
I(k, l, h_1, \ls, h_r) - I_0(k, l, h_1, \ls, h_r) = \phi \text{\ \ \ if \ } l=h_1\geq 1.
\end{align} 
Hence we have 
\begin{align*}
& \Dq(\Phi^t_{0} - \Phi^t_{1}) \\
& = \f{1}{z} \f{x_3}{x_2}
\sum_{{{k \geq l+\sum_{j=1}^{r}h_j} \atop {l \geq h_1 \geq \cs \geq h_r \geq 0}} \atop{l \geq 1 }}
\left\{ G^t(k, l, h_1, h_2, \ls, h_r ; z)
- G^t_{0}(k, l, h_1, h_2, \ls, h_r ; z)\right\}\\
& \qquad \times 
x_1^{k-l-\sum_{j=1}^{r}h_j} x_2^{l-h_1} x_3^{h_1-h_2} \cs x_{r+1}^{h_{r-1} - h_{r}} x_{r+2}^{h_r}\\
& = \f{x_{3}}{z x_{2}} (\Phi^t -1-\Phi^t_{0}). 
\end{align*}

We move on to show the last equation. First decompose $\Phi^t - \Phi^t_0$ into three parts (for $l=0,\ l=1$ and $l \geq 2$): 
\begin{align*}
 \Phi^t - \Phi^t_0 
& = \sum_{{k \geq l+\sum_{j=1}^{r}h_j} \atop {l \geq h_1 \geq \cs \geq h_r \geq 0}}
\left\{ G^t(k, l, h_1, \ls, h_r ; z) - G^t_{0}(k, l, h_1, \ls, h_r ; z)\right\}\\
& \qquad \times 
x_1^{k-l-\sum_{j=1}^{r}h_j} x_2^{l-h_1} x_3^{h_1-h_2} \cs x_{r+1}^{h_{r-1} - h_{r}} x_{r+2}^{h_r}\\
& = 1 + G^t(1, 1, 0, \ls, 0; z)x_2 \\
& \quad + \sum_{{{k \geq l+\sum_{j=1}^{r}h_j} \atop {l > h_1 \geq \cs \geq h_r \geq 0}}\atop{l \geq 2}}
\left\{ G^t(k, l, h_1, \ls, h_r ; z) - G^t_{0}(k, l, h_1, \ls, h_r ; z)\right\} \\
& \qquad \times
x_1^{k-l-\sum_{j=1}^{r}h_j} x_2^{l-h_1} x_3^{h_1-h_2} \cs x_{r+1}^{h_{r-1} - h_{r}} x_{r+2}^{h_r}. 
\end{align*}
Note that we used \eqref{II0} in the third term. 
Applying Lemma \ref{lem3} (iii), we obtain
\begin{align*}
\Dq( \Phi^t - \Phi^t_0) 
& =  \f{x_2}{1-z} + \left( \f{t}{z} + \f{1}{1-z} \right)  x_2 
\sum_{{{k \geq l+\sum_{j=1}^{r}h_j} \atop {l \geq h_1 \geq \cs \geq h_r \geq 0}} \atop {l \geq 1}}
G^t(k, l, h_1, \ls, h_r ; z)\\
& \qquad \times x_1^{k-l-\sum_{j=1}^{r}h_j} x_2^{l-h_1} x_3^{h_1-h_2} \cs x_{r+1}^{h_{r-1} - h_{r}} x_{r+2}^{h_r} \\
& = \f{x_2}{1-z} + \left( \f{t}{z} + \f{1}{1-z} \right)x_2 \left( \Phi^t -1 \right). 
\end{align*}
\end{proof}

Now we define another $q$-difference operator $\Tq$ : 
$$(\Tq f)(z) := z (\Dq f)(z) = \f{f(z) -f(qz)}{1-q}. $$
\begin{cor}\label{cor5}
\begin{itemize}
\item[(i)] We have
$$\left\{ \begin{array}{l}
\Tq \Phi^t_{r-1} = \d x_1 \Phi^t_{r-1} + \f{x_{r+2}}{x_{r+1}}(\Phi^t_{r-2} -\Phi^t_{r-1}),  \\
\Tq (\Phi^t_{j} - \Phi^t_{j+1}) =\d  \f{x_{j+3}}{x_{j+2}} (\Phi^t_{j-1} -\Phi^t_{j}), \qquad j=1, 2, \ls, r-2, \\
\Tq (\Phi^t_{0} - \Phi^t_{1}) = \d \f{x_{3}}{x_{2}} (\Phi^t -\Phi^t_{0}-1),  \\
\Tq (\Phi^t - \Phi^t_{0}) = \d \f{z}{1-z}x_2 + \f{z+(1-z)t}{1-z} x_2 (\Phi^t-1).  
\end{array}
\right. $$
\item[(ii)] Let $y_0=\Phi^t_{r-1}$, $y_j = \Phi^t_{r-1-j} - \Phi^t_{r-j}\ (j= 1, 2, \ls, r-1)$, $y_r = \Phi^t -\Phi^t_0$. Then we have
$$y_j = \f{x_{r+2-j}}{x_{r+2}}\Tq^j y_0 -\f{x_1 x_{r+2-j}}{x_{r+2}}\Tq^{j-1}y_0 + \delta_{j, r}$$
for $j=1, 2, \ls, r$, where $\delta$ stands for the Kronecker's delta. 
\item[(iii)] The function $y_0 = \Phi^t_{r-1}(z)$ satisfies the following q-difference equation 
\begin{align*}
& \left\{ 
\Tq^{r+1} -(x_1+tx_2)\Tq^r - t \sum_{j=0}^{r-1} (x_{r+2-j} - x_1 x_{r+1-j})\Tq^j \right. \\
&\quad  \left. -z \left( \Tq^{r+1} + ((1 - t) x_2 - x_1)  \Tq^r + (1-t) \sum_{j=0}^{r-1}(x_{r+2-j} - x_1 x_{r+1-j}) \Tq^j
\right)
\right\}y_0\\
&  = x_{r+2}z. 
\end{align*}
\end{itemize}
\end{cor}
\begin{proof}
(i) is easily shown by using $\Tq = z \Dq$. 

(ii) The relations in (i) can be written as 
\begin{align}\label{y}
\left\{ \begin{array}{l}
\Tq y_0 = \d x_1 y_0 + \f{x_{r+2}}{x_{r+1}}y_1, \\
\Tq y_j =\d  \f{x_{r+2-j}}{x_{r+1-j}}y_{j+1} -\f{x_{r+2-j}}{x_{r+1-j}} \delta_{r-1, j},  \qquad j=1, 2, \ls, r-1 \\
\Tq y_r= \d \f{z}{1-z}x_2 + \f{z+(1-z)t}{1-z} x_2 (y_0 + y_1 + \cs + y_r -1). 
\end{array}
\right. 
\end{align}
By the second relation of \eqref{y}, we have
\begin{align}\label{yj}
y_{j+1} = \f{x_{r+1-j}}{x_{r+2-j}} \Tq y_j + \delta_{j, r-1} , \quad j=1, 2, \ls, r-1.
\end{align} 
Hence we have 
$$ y_j=\f{x_{r+2-j}}{x_{r+1}} \Tq^{j-1} y_1,\quad j = 2, \ls, r-1. $$
Since $y_1 = \f{x_{r+1}}{x_{r+2}} (\Tq y_0 -x_1 y_0)$,  we have 
$$y_{j} = \f{x_{r+2-j}}{x_{r+2}} \left( \Tq^j y_0 - x_1 \Tq^{j-1} y_0\right)$$
for $j= 2, \ls, r-1$. 
This is also valid for $j=1$. By \eqref{yj} we have 
$$y_r = \f{x_2}{x_3} \Tq y_{r-1} + 1. $$
This completes the proof. 
%

(iii) By the difference equation (ii) for $j=1, 2, \ls, r$, we have 
\begin{align}\label{y's}
y_0 + y_1 + \cs +y_r -1 & = y_0 + \sum_{j=1}^r\f{x_{r+2-j}}{x_{r+2}} \Tq^j y_0 - \sum_{j=1}^{r} \f{x_1 x_{r+2-j}}{x_{r+2}} \Tq^{j-1}y_0 \nonumber \\
& = \sum_{j=0}^{r-1}\f{x_{r+2-j}-x_1 x_{r+1-j}}{x_{r+2}} \Tq^j y_0 + \f{x_{2}}{x_{r+2}} \Tq^{r}y_0.
\end{align} 
On the other hand,  $y_r = \f{x_2}{x_{r+2}} \Tq^ry_0 - \f{x_1x_2}{x_{r+2}} \Tq^{r-1}y_0 + 1$ which is the case of $j=r$ in (ii). 
Applying $\Tq$ to this equation and combining it with \eqref{y's} and the third relation of \eqref{y}, we have 
\begin{align*}
& \f{x_2}{x_{r+2}} \Tq^{r+1} y_0 - \f{x_1x_2}{x_{r+2}} \Tq^{r}y_0 \\
& = \f{z}{1-z}x_2 + \f{z+(1-z)t}{1-z}x_2 \left( \sum_{j=0}^{r-1}\f{x_{r+2-j}-x_1 x_{r+1-j}}{x_{r+2}} \Tq^j y_0 + \f{x_{2}}{x_{r+2}} \Tq^{r}y_0 \right). 
\end{align*}   
Hence we obtain
\begin{align*}
& (1-z) \Tq^{r+1} y_0 - x_1 (1-z) \Tq^{r}y_0 \\
& =  x_{r+2}z +  \{ z+(1-z)t \} \left( \sum_{j=0}^{r-1} (x_{r+2-j}-x_1 x_{r+1-j}) \Tq^j y_0 + x_{2} \Tq^{r}y_0 \right) 
\end{align*}
and the conclusion.    
\end{proof}
%
\subsection{Relation between $\Phi^t_j(z)$ and the basic hypergeometric functions}
First, we prove the following relation between $\Phi^t_{r-1}(z)$ and the basic hypergeometric function ${}_{r+2} \phi _{r+1}$. 
\begin{thm}\label{pr-1}
Let $r$ be a positive integer. 
We have
\begin{align*}
\Phi^t_{r-1} & = \f{x_{r+2} z}{1-(x_1 + t x_2) - t \sum_{j=0}^{r-1}(x_{r+2-j} - x_1 x_{r+1-j})} \\
& \qquad \times  {}_{r+2} \phi _{r+1} 
 \left[ 
 \begin{array}{c}
 q, a_1, \ls, a_{r+1} \\
 b_1, \ls, b_{r+1}
 \end{array}
 ; q, \f{b_1 \cs b_{r+1}}{q^{r+1} a_1 \cs a_{r+1}} z
 \right],  
\end{align*}
where $a_i, b_i$'s are given by \eqref{ai} and \eqref{bi}. 
\end{thm}
\noindent
For the proof, we need the next lemmas. 
%
\begin{lem}\label{l10}
\begin{itemize}
\item[(i)] Let $r$ be a positive integer and 
$b_j \neq q^{-m}$ for $m = 0, 1, \ls$ and $j = 1, 2, \ls, r$. 
For any non-negative integer $n$, 
\begin{align*}
&\Dq^n {}_{r+1}\phi_{r}
\left[ 
\begin{array}{c}
a_1, \ls, a_{r+1}\\
b_1, \ls, b_r
\end{array} 
; q, az \right]\\
&= \f{(a_1, \ls, a_{r+1}; q)_n\ a^n}{(b_1, \ls, b_r; q)_n (1-q)^n}
{}_{r+1}\phi_{r}
\left[ 
\begin{array}{c}
a_1q^n, \ls, a_{r+1}q^n\\
b_1q^n, \ls, b_rq^n
\end{array} 
; q, az \right]. 
\end{align*}
\item[(ii)] 
Let $r$ be a positive integer and $a_1 \cs a_{r+1} b_1 \cs b_{r} \neq 0$. Let $b_j \neq q^{-m}$ for $m = 0, 1, \ls$ and $j = 1, 2, \ls, r$. 
Then  
$u(z) = {}_{r+1}\phi_{r}\left[ \begin{array}{c}a_1, \ls, a_{r+1}\\ b_1, \ls, b_r \end{array} ; q, z \right]$ 
satisfies
\begin{align*}
& \Tq \left( \Tq + \f{q-b_1}{b_1(1-q)} \right) \cs \left( \Tq + \f{q-b_r}{b_r(1-q)} \right) u(z) \\
& = \f{a_1 a_2 \cs a_{r+1}}{b_1 b_2 \cs b_r} q^r z \left( \Tq + \f{1-a_1}{a_1(1-q)} \right) \cs \left( \Tq + \f{1-a_{r+1}}{a_{r+1}(1-q)} \right) u(z),  
\end{align*}
and $v(z) = {}_{r+1}\phi_{r}\left[ \begin{array}{c}a_1, \ls, a_{r+1}\\ b_1, \ls, b_r \end{array} ; q, \f{b_1 \cs b_r}{a_1 \cs a_{r+1} q^r}z \right]$ 
satisfies
\begin{align*}
& \Tq \left( \Tq + \f{q-b_1}{b_1(1-q)} \right) \cs \left( \Tq + \f{q-b_r}{b_r(1-q)} \right) v(z) \\
& = z \left( \Tq + \f{1-a_1}{a_1(1-q)} \right) \cs \left( \Tq + \f{1-a_{r+1}}{a_{r+1}(1-q)} \right) v(z). 
\end{align*}
\item[(iii)] For a non-negative integer $n$, 
$$\Tq^n = \sum_{m=0}^{n} S_q(n, m) z^m \Dq^m. $$ 
\end{itemize}
\end{lem}
\begin{proof}
(i) and (iii) are shown by induction on $n$. (ii) is derived from a consequence of \cite[Exercise 1.31]{G}. 
\end{proof}
\begin{lem}\label{LEM}
Suppose $f(z)$ is holomorphic at $z=0$. 
If $(\Tq -1) (f(z)) = 0$, 
we have $f(z) = c z$ for certain constant $c$.  
\end{lem}
\begin{proof}
By assumption we have $f(qz) = q f(z)$. 
Substituting $f(z) = \sum_{n=0}^{\infty}a_n z^n $ into this equality, we have $a_0 = a_2 = a_3 = \cs = 0$ while $a_1$ is arbitrary. 
\end{proof}
\begin{proof}[Proof of Theorem \ref{pr-1}]
Put 
\begin{align*}
\mathcal{L} & := 
\Tq^{r+1} -(x_1+tx_2)\Tq^r - t \sum_{j=0}^{r-1} (x_{r+2-j} - x_1 x_{r+1-j})\Tq^j \\
&\quad  \left. -z \left( \Tq^{r+1} + ((1 - t) x_2 - x_1) \Tq^r + (1-t) \sum_{j=0}^{r-1}(x_{r+2-j} - x_1 x_{r+1-j}) \Tq^j
\right)\right.
\end{align*}
By Corollary \ref{cor5} (iii), we have 
\begin{align}\label{L1}
\mathcal{L} (y_0) = x_{r+2} z, 
\end{align}
where $y_0=\Phi_{r-1}^t$. 

By the way, according to the definition of $\alpha_i$'s and $\beta_i$'s (see \S \ref{Intro}), we have 
$$\mathcal{L} = (\Tq + \beta_1) \cs (\Tq + \beta_{r+1}) - z (\Tq + \alpha_1) \cs (\Tq + \alpha_{r+1}). $$
Put
\begin{align*}
\widetilde{\mathcal{L}}& :=
\Tq \left( \Tq+\f{1}{q}( 1+ \beta_1) \right) \cs \left( \Tq+\f{1}{q} (1 + \beta_{r+1}) \right) \\
& \quad - z q \left( \Tq+\f{1}{q} \right) \left( \Tq+\f{1}{q}(1 + \alpha_1) \right) \cs \left( \Tq+\f{1}{q}(1 + \alpha_{r+1}) \right). 
\end{align*}
Since $\Tq (z f(z)) = z(q \Tq +1) (f(z))$,  we find
\begin{align*}
(\Tq -1) \mathcal{L} (z f(z)) & = z q^{r+2}\widetilde{\mathcal{L}} (f(z)). 
\end{align*}
By Lemma \ref{l10} (ii), we have
$$\widetilde{\mathcal{L}} (v(z)) = 0, $$
where 
$v(z) = {}_{r+2}\phi_{r+1}\left[ \begin{array}{c} q, a_1, \ls, a_{r+1}\\ b_1, \ls, b_{r+1} \end{array} ; q, \f{b_1 \cs b_{r+1}}{a_1 \cs a_{r+1} q^{r+1}} z \right]$, 
and hence
$$(\Tq -1) \mathcal{L} (z v(z)) = 0 . $$
In particular, by Lemma \ref{LEM}, we have 
\begin{align}\label{L2}
\mathcal{L} (z v(z)) = c z
\end{align}
for certain constant $c$. 
From \eqref{L1} and \eqref{L2}, we have 
$$\mathcal{L}(c y_0 - x_{r+2} z v(z)) = 0. $$
Since both $y_0$ and $z v(z)$ have no constant terms, we have 
$$c y_0 = x_{r+2} z v(z)$$
due to \cite[Lemma 3.4]{Li}.  

Put $y_0 = Az + o(z^2)$ and let us determine the constant $A$. 
Notice that $\Tq z^n = [n] z^n$ for any positive integer $n$. Then from 
the difference equation of Corollary \ref{cor5} (iii), we get
$$\left\{ 1-(x_1+tx_2) - t \sum_{j=0}^{r-1}(x_{r+2-j} - x_1 x_{r+1-j})\right\} Az + o(z^2) = x_{r+2} z, $$
and hence 
$$A \left( = \f{x_{r+2}}{c} \right)=\f{x_{r+2}}{1-(x_1+tx_2) - t \sum_{j=0}^{r-1}(x_{r+2-j} - x_1 x_{r+1-j})} . $$
This completes the proof. 
\end{proof}
\noindent

Secondly, we prove the following relation for $\Phi_j^t(z)$. 
\begin{thm}\label{lem7}
Put $y_0 = \Phi^t_{r-1}$. For $-1 \leq j \leq r-1$, we have  
$$\Phi^t_{j}(z) = \f{1}{x_{r+2}} \sum_{i=0}^{r-1-j} A^{(j)}_{i} z^i \Dq^i y_0 + \delta_{j, -1}, $$
where 
$$A^{(j)}_i = \sum_{m=i}^{r-1-j} (x_{r+2-m} - x_1 x_{r+1-m}) S_q(m, i) + x_1 x_{j+2} S_q(r-1-j, i)$$
for $0 \leq i \leq r-1-j$. 
\end{thm}
\begin{proof}
If $y_0 = \Phi^t_{r-1}$, $y_1 = \Phi^t_{r-2} - \Phi^t_{r-1}$, \ls, $y_{r-1-j} = \Phi^t_{j} - \Phi^t_{j+1}$, 
we have 
$$\Phi^t_j= y_0 + \sum_{m=1}^{r-1-j} y_m. $$
Using Corollary \ref{cor5} (ii), 
we have 
\begin{align*}
\Phi^t_j 
& = y_0 + \sum_{m=1}^{r-1-j} \left(  \f{x_{r+2-m}}{x_{r+2}}\Tq^m y_0 -\f{x_1 x_{r+2-m}}{x_{r+2}}\Tq^{m-1}y_0 \right) + \delta_{j, -1} \\
& = \sum_{m=0}^{r-1-j} \f{x_{r+2-m} - x_1 x_{r+1-m}}{x_{r+2}}\Tq^m y_0 +\f{x_1 x_{j+2}}{x_{r+2}}\Tq^{r-1-j}y_0 + \delta_{j, -1}. \\
\end{align*}
By Lemma \ref{l10} (iii), 
\begin{align*}
\Phi^t_j 
&  = \sum_{m=0}^{r-1-j} \f{x_{r+2-m} - x_1 x_{r+1-m}}{x_{r+2}} \sum_{i=0}^{m} S_q(m, i) z^i \Dq^i y_0 \\
& \quad +\f{x_1 x_{j+2}}{x_{r+2}} \sum_{i=0}^{r-1-j} S_q(r-1-j, i) z^i \Dq^i y_0 + \delta_{j, -1} \\
& = \sum_{i=0}^{r-1-j} \sum_{m=i}^{r-1-j} \f{x_{r+2-m} - x_1 x_{r+1-m}}{x_{r+2}} S_q(m, i) z^i \Dq^i y_0\\
& \quad +\f{x_1 x_{j+2}}{x_{r+2}} \sum_{i=0}^{r-1-j} S_q(r-1-j, i) z^i \Dq^i y_0 + \delta_{j, -1} \\
& =  \f{1}{x_{r+2}} \sum_{i=0}^{r-1-j} A^{(j)}_{i} z^i \Dq^i y_0 + \delta_{j, -1}. 
\end{align*}
\end{proof}
Finally, we show the following relation between $\Phi_j^t$ and the basic hypergeometric function ${}_{r+2} \phi_{r+1}$. 
\begin{thm}\label{thm9}
For $-1 \leq j \leq r-1$,  
we have
\begin{align*}
\Phi^t_{j}(z) & = \f{1}{1-(x_1 + t x_2) - t \sum_{i=0}^{r-1}(x_{r+2-i} - x_1 x_{r+1-i}) }\\
& \quad \times \sum_{i=0}^{r-1-j} \widetilde{A^{(j)}_i} B_i z^{i+1} 
{}_{r+2}\phi_{r+1}\left[ \begin{array}{c} q^{i+1}, a_1q^i, \ls, a_{r+1}q^i \\ b_1q^i, \ls, b_{r+1}q^i \end{array} ; q, \f{b_1 \cs b_{r+1}}{q^{r+1} a_1 \cs a_{r+1}} z \right]
+\delta_{j, -1}, 
\end{align*}
where 
$$\widetilde{A^{(j)}_i} = \sum_{m=i}^{r-1-j} (x_{r+2-m} -x_1 x_{r+1-m}) S_q(m+1, i+1) + x_1 x_{j+2} S_q(r-j, i+1)$$
for $0 \leq i \leq r-1-j$, and 
$$B_i=\f{(q, a_1, \ls, a_{r+1}; q)_i}{(b_1, \ls, b_{r+1}; q)_i} \left(\f{b_1 \cs b_{r+1}}{q^{r+1} (1-q) a_1 \cs a_{r+1}} \right)^i$$
for $0 \leq i \leq r$. 
\end{thm}
\noindent
For the proof, we need the next lemma. 
\begin{lem}\label{l13}
For any non-negative integer $i$, 
\begin{align*}
\Dq^i (z f(z)) = [i] \Dq^{i-1}(f(z)) + q^i z \Dq^i(f(z)). 
\end{align*}
\end{lem}
\begin{proof}
We prove this lemma by induction on $i$. 
If $i=0$, it holds because both sides become $zf(z)$. 
If $i = 1$, by the definition of $\Dq$, 
\begin{align*}
\Dq(zf(z)) & = \f{z f(z) - qz f(qz)}{(1-q)z} \\
& =\f{zf(z) - qz f(z) + qz f(z) - qz f(qz)}{(1-q)z} \\
& = f(z) + qz \Dq(f(z)). 
\end{align*}
Therefore the identity holds. 
Assume that the identity holds for $i$ ($i \geq 1$).  
By the induction hypothesis, 
\begin{align*}
\Dq^{i+1}(zf(z)) & = \Dq \Dq^{i}(zf(z)) = \Dq \left(  [i] \Dq^{i-1}(f(z)) + q^i z \Dq^i(f(z)) \right)\\
& =  [i] \Dq^{i}(f(z)) + q^i \Dq ( z \Dq^i(f(z))) \\
& =  [i] \Dq^{i}(f(z)) + q^i \Dq^i(f(z)) +  q^{i+1} z \Dq^{i+1}(f(z)). 
\end{align*}
By $[i] + q^i = \f{1-q^i}{1-q} + q^i = \f{1-q^{i+1}}{1-q} = [i+1]$, we have 
$$\Dq^{i+1}(zf(z)) = [i+1] \Dq^{i}(f(z)) + q^{i+1} z \Dq^{i+1}(f(z)).$$
\end{proof}
\begin{proof}[Proof of Theorem \ref{thm9}]
Set 
$$A=\f{1}{1-(x_1 + t x_2) - t \sum_{i=0}^{r-1}(x_{r+2-i} - x_1 x_{r+1-i}) } $$
and 
$$v(z) = {}_{r+2}\phi_{r+1}\left[ \begin{array}{c} q, a_1, \ls, a_{r+1} \\ b_1, \ls, b_{r+1} \end{array} ; q, \f{b_1 \cs b_{r+1}}{q^{r+1}a_1 \cs a_{r+1}} z \right]. $$
Applying Theorem \ref{pr-1} to Theorem \ref{lem7},  
we have 
$$\Phi^t_{j}(z) = A \sum_{i=0}^{r-1-j} A^{(j)}_{i} z^i \Dq^i(zv(z)) + \delta_{j, -1}.$$
By Lemma \ref{l13}, 
\begin{align*}
\sum_{i=0}^{r-1-j} A^{(j)}_{i} z^i \Dq^i(zv(z))
 & = \sum_{i=0}^{r-1-j} A^{(j)}_{i} z^i [i] \Dq^{i-1}(v(z)) + \sum_{i=0}^{r-1-j} A^{(j)}_{i} z^{i+1} q^i \Dq^i(v(z)) \\
 & = \sum_{i=0}^{r-1-j} \left( A^{(j)}_{i+1} [i+1]  + A^{(j)}_{i}q^i \right) z^{i+1} \Dq^i(v(z)),  
\end{align*}
where $A^{(j)}_{r-j} = 0$. 
According to the definition of $A^{(j)}_{i}$, $\widetilde{A^{(j)}_{i}}$ and the $q$-Stirling number $S_q$, 
$$A^{(j)}_{i+1} [i+1]  + A^{(j)}_{i}q^i = \widetilde{A^{(j)}_i}. $$ 
Also by Lemma \ref{l10} (i), we obtain
$$\Dq^i(v(z)) = B_i \ {}_{r+2}\phi_{r+1}\left[ \begin{array}{c} q^{i+1}, a_1q^i, \ls, a_{r+1}q^i \\ b_1q^i, \ls, b_{r+1}q^i \end{array} ; q, \f{b_1 \cs b_{r+1}}{q^{r+1}a_1 \cs a_{r+1}}z \right]. $$
This completes the proof.
\end{proof}
For $i=0, 1, \ls, r-1$, we set  
$$v_i = {}_{r+2}\phi_{r+1}\left[ \begin{array}{c} q^{i+1}, a_1q^i, \ls, a_{r+1}q^i \\ b_1q^i, \ls, b_{r+1}q^i \end{array} ; q, \f{b_1 \cs b_{r+1}}{q^{r}a_1 \cs a_{r+1}} \right] $$
and 
$$\widetilde{B_i} = B_i q^{i+1} = q \f{(q, a_1, \ls, a_{r+1}; q)_i}{(b_1, \ls, b_{r+1}; q)_i} \left( \f{b_1 \cs b_{r+1}}{q^r(1-q)a_1 \cs a_{r+1}} \right)^i. $$
According to Theorem \ref{thm9}, we obtain the following corollary. 
\begin{cor}\label{Cor}
The identity 
$$\Phi^t_j(q) = \f{1}{c} \sum_{i=0}^{r-1-j} \widetilde{A^{(j)}_i} \widetilde{B_i} v_i,\quad j=0, 1, \ls, r-1$$
holds, where  
\begin{align*}
c= 1-(x_1+tx_2) - t\sum_{i=0}^{r-1} (x_{r+2-i} - x_1 x_{r+1-i}).  
\end{align*}
Or equivalently we have
$$
\begin{pmatrix} \Phi^t_0(q) \\ \Phi^t_1(q) \\ \vdots \\ \Phi^t_{r-1}(q)  \end{pmatrix}
= \f{1}{c} A \begin{pmatrix} \widetilde{B_0} v_0 \\ \widetilde{B_1} v_1 \\ \vdots \\ \widetilde{B_{r-1}} v_{r-1} \end{pmatrix}$$
with
$$A = \begin{pmatrix} 
\widetilde{A^{(0)}_0} &   \widetilde{A^{(0)}_1} & \cs &  \widetilde{A^{(0)}_{r-2}} &  \widetilde{A^{(0)}_{r-1}} \\
\widetilde{A^{(1)}_0} &   \widetilde{A^{(1)}_1} & \cs &  \widetilde{A^{(1)}_{r-2}} &  \\
\vdots & \vdots & \iddots & & \\
\widetilde{A^{(r-2)}_0} &   \widetilde{A^{(r-2)}_1} & & & \\
\widetilde{A^{(r-1)}_0} & & & & \end{pmatrix}. $$
\end{cor}
\subsection{$t$-$q$MZVs as special values of $t$-$q$MPLs}
The $t$-$q$MPLs are related to $t$-$q$MZVs as follows. 
\begin{lem}\label{lem1}
For positive integers $k_1, k_2, \ls, k_l$ with $k_1 \geq 2$, we have 
$$Li^{t}_{\k; q}(q) = \sum_{a_1=2}^{k_1}\sum_{a_2=1}^{k_2} \cs \sum_{a_l=1}^{k_l}\binom{k_1-2}{a_1-2}\prod_{j=2}^{l}\binom{k_j-1}{a_j-1} (1-q)^{\sum_{i=1}^{l}(k_i-a_i)}\tq(a_1, a_2, \ls, a_l). $$
\end{lem}
For its proof we here introduce the algebraic setup of $t$-$q$MZVs (see \cite{W} for details). 
Let $\h$ be a formal variable.  Denote by $\Hht = \Q[\h, t] \langle x,y \rangle$ the non-commutative polynomial algebra over $\Q[\h, t]$ in two indeterminates $x$ and $y$, 
and by $\Hht^{1}$ and $\Hht^{0}$ its subalgebras $\Q[\h, t]+ \Hht y$ and $\Q[\h, t]+ x\Hht y$, respectively. 
Put $z_j = x^{j-1}y\ (j \geq 1)$.  We define the weight and the depth of a word $u=z_{k_1} z_{k_2} \cdots z_{k_l}$ by ${\rm wt}(u)=k_1+k_2+\cdots+k_l$ and ${\rm dep}(u)=l$, respectively. 
Define the $\Q[\h, t]$-linear map $\widehat{Z}_q^t : \Hht^{0} \longrightarrow \Q[\h, t][[q]]$ by $\widehat{Z}_q^t(1) = 1$ and 
$$\widehat{Z}_q^t(z_{k_1} z_{k_2} \cs z_{k_l}) = \tq(k_1, k_2, \ls, k_l) \ \ \ (k_1 \geq 2). $$
We also define the substitution map  
$f : \Q[\h, t][[q]] \longrightarrow \Q[t][[q]]$ by $f : \h \longmapsto 1-q$ and set 
$$Z_q^t = f \circ \widehat{Z}_q^t. $$  
%
We let  
$S^t= R_y \gamma^tR_y^{-1}, $
where $\gamma^t$ denotes the automorphism on $\Hht$ characterized by 
$\gamma^t(x) = x,\ \gamma^t(y) = tx + y$   
and the maps $R_y$ and $L_x$ are $\Q$-linear maps, called right-multiplication of $y$ and left--multiplication of $x$, defined respectively by $R_y(w) = wy$ and $L_x(w) = xw$ for any $w \in \Hht$. 
Also we put
$S_\h^t=R_y \gamma_\h^t R_y^{-1}$,  
where $\gamma_\h^t$ denotes the automorphism on $\Hht$ characterized by 
$\gamma_\h^t(x) = x,\ \gamma_\h^t(y) = tx + y + \h t$,   
and $z_{\k} : = z_{k_1}z_{k_2} \cs z_{k_l}$ for $\k = (k_1, \ls, k_l)$. 

\begin{proof}[Proof of Lemma \ref{lem1}] 
This identity holds when $t=0$ according to (7) in Okuda-Takeyama\cite{OT}. 
Let $\mathcal{L}_{\h}$ be the $\Q[t]$-linear map given by 
$$\mathcal{L}_{\h}(z_{k_1}z_{k_2} \cs z_{k_l}):=\sum_{a_1=2}^{k_1}\sum_{a_2=1}^{k_2} \cs \sum_{a_l=1}^{k_l}\binom{k_1-2}{a_1-2}\prod_{j=2}^{l}\binom{k_j-1}{a_j-1} \h^{\sum_{i=1}^{l}(k_i-a_i)}z_{a_1}z_{a_2} \cs z_{a_l}. $$
Let $f_{\h}$ denote the automorphism on $\Hht$ characterized by 
$f_\h(x) = x + \h,\ f_\h(y) = y.$
Then we have $\mathcal{L}_{\h}=L_x f_\h L_x^{-1}$.  
By definition of $S_\h^t$, we find that  
\begin{align*}
Z^t_q =  Z_q^0 S_\h^t, 
\end{align*}
where $Z_q^0(w) : = Z_q^t  (w) |_{t=0}$ ($w \in \Hht^0$).
Hence we have 
\begin{align*}
& Z_q^0 S_\h^t \mathcal{L}_{\h}(z_{k_1}z_{k_2} \cs z_{k_l}) \\
 & = Z_q^t \mathcal{L}_{\h}(z_{k_1}z_{k_2} \cs z_{k_l}) \\
 & = \sum_{a_1=2}^{k_1}\sum_{a_2=1}^{k_2} \cs \sum_{a_l=1}^{k_l}\binom{k_1-2}{a_1-2}\prod_{j=2}^{l}\binom{k_j-1}{a_j-1} (1-q)^{\sum_{i=1}^{l}(k_i-a_i)}\tq(a_1, a_2, \ls, a_l). 
\end{align*}
We also have 
\begin{align*}
\mathcal{L}_{\h} S^t (z_{k_1}z_{k_2} \cs z_{k_l}) 
& = \mathcal{L}_{\h} (x^{k_1-1}(tx+y) \cs x^{k_{l-1}-1}(tx+y) x^{k_l-1}y) \\
& = \mathcal{L}_{\h} \left( {\sum_{{\p}}}' t^{l - \rm{dep}({\p})} z_\p \right)
= {\sum_{{\p}}}' t^{l - \rm{dep}({\p})} \mathcal{L}_{\h} (z_\p)\\
& \stackrel{Z_q^0}{\mapsto} {\sum_{{\p}}}' t^{l - \rm{dep}({\p})} Li_{\p;q}(q) \qquad (\text{by using Lemma \ref{lem1} when } t=0)\\
&= Li_{\k;q}^t(q).
\end{align*}
Moreover we have
\begin{align*}
\mathcal{L}_{\h} S^t & = L_x f_\h L_x^{-1} R_y \gamma^t R_y^{-1} \\
                                  & = L_x f_\h R_y L_x^{-1} \gamma^t R_y^{-1} \qquad (\text{since } [L_x^{-1}, R_y] = 0)\\
                                  & = L_x R_y f_\h \gamma^t L_x^{-1} R_y^{-1} \qquad (\text{since } [f_\h, R_y] = 0, [L_x^{-1}, \gamma^t] = 0)\\
                                  & = R_y L_x \gamma_\h^t f_\h R_y^{-1} L_x^{-1} \qquad (\text{since } [L_x, R_y] = 0, f_\h \gamma^t = \gamma_\h^t f_\h)\\
                                  & = R_y \gamma_\h^t L_x R_y^{-1} f_\h L_x^{-1} \qquad (\text{since } [f_\h, R_y^{-1}] = 0, [L_x, \gamma_\h^t] = 0)\\
                                  & = R_y \gamma_\h^t R_y^{-1} L_x f_\h L_x^{-1} \qquad (\text{since }  [L_x, R_y^{-1}] = 0)\\
                                  & = S_\h^t \mathcal{L}_{\h},                                    
\end{align*}
and hence the conclusion. 
\end{proof}
\section{Proof of main theorem}\label{SumtqMZV}
Combining the following four propositions (Propositions \ref{lem10}, \ref{Prop}, \ref{propAj}, \ref{lem13}), 
we obtain our main theorem. 
\subsection{Relation between $\Psi^t_j$ and $\Phi^t_0(q)$}
For any non-negative integers $k, l, h_1, \ls, h_r$ $(r \in \N)$ and $j=0, 1, \ls, r-1$,  
we set 
\begin{align*}
\xi^t_j&(k, l, h_1, \ls, h_r) = \sum_{\k \in I_j(k, l, h_1, \ls, h_r)} \tq(\k), \\
\Psi^t_j& =\Psi^t_j(u_1, u_2, \ls, u_{r+2}) \\
              & = \sum_{k, l, h_1, \ls, h_r \geq 0} \xi^t_j(k, l, h_1, \ls, h_r)
                     u_1^{k-l-\sum_{i=1}^{r}h_i} u_2^{l-h_1} u_3^{h_1 - h_2} \cs u_{r+1}^{h_{r-1} - h_{r}} u_{r+2}^{h_r}.  
\end{align*}
We set 
$$X_p = \sum_{i=p}^{r} { i-1 \choose p-1} ((1-q)x_1)^{i-p} \f{x_1^{r+1-i} x_{i+1} - x_{r+2}}{x_{r+2}} + (1-(1-q)x_1)^{-p},\ p=1, \ls, r+1, $$
and 
$$Y_p = \sum_{i=p}^{r} { i-2 \choose p-2} ((1-q)x_1)^{i-p} \f{x_1^{r+1-i} x_{i+1} - x_{r+2}}{x_{r+2}} + (1-(1-q)x_1)^{1-p},\ p=2, \ls, r, $$
where $x_i$'s are given by \eqref{x1}, \eqref{xj}. 
\begin{prop}\label{lem10}
We have 
$$\Psi^t_{0} = \f{\Phi^t_0(q)}{1-(1-q)x_1} + \underline{X} M_1 
\begin{pmatrix}
\Phi^t_0(q) - \Phi^t_1(q)\\
\Phi^t_0(q) - \Phi^t_2(q)\\
\vdots\\
\Phi^t_0(q) - \Phi^t_{r-1}(q)
\end{pmatrix}, $$
where 
$$\underline{X} = \left( X_2-\f{Y_2}{1-(1-q)x_1}, \ls, X_r-\f{Y_r}{1-(1-q)x_1} \right)$$
and
$$M_1=x_{r+2} 
\begin{pmatrix} \f{1}{x_1^{r-1}} & & & \\ & \f{1}{x_1^{r-2}} & & \\ & & \ddots & \\ & & & \f{1}{x_1} \end{pmatrix}
T 
\begin{pmatrix} \f{1}{x_3} & & & \\ & \f{1}{x_4} & & \\ & & \ddots & \\ & & & \f{1}{x_{r+1}} \end{pmatrix}
\begin{pmatrix} 1 & & & \\ -1&1 & & \\ & \ddots& \ddots & \\ & &-1& 1 \end{pmatrix} $$
with the lower triangular matrix $T=(t_{ij})_{(r-1) \times (r-1)}$ defined by 
\begin{align*}
t_{ij}= 
\left\{ \begin{array}{ll} {i-1 \choose j-1} (q-1)^{i-j}, & i \geq j, \\
 0, & i < j. \end{array} \right. \end{align*}
\end{prop}
\noindent
To prove this proposition, we show three lemmas. 

\def\diag{\operatorname{diag}}
Set $$Z_{pj}=\sum_{i=p}^{j}\left({i-2\atop p-2}\right)((1-q)x_1)^{i-p}
\frac{x_1^{r+1-i}x_{i+1}}{x_{r+2}},\;\;2\leq p\leq j\leq r.$$ Let
$W$ and $D$ be $r\times r$ matrices defined by
$$W=\begin{pmatrix}
Y_2 & Y_3 & Y_4 & \cdots & Y_r & 1\\
Y_2-Z_{22} & Y_3 & Y_4 & \cdots & Y_r & 1 \\
Y_2-Z_{23} & Y_3-Z_{33} & Y_4 & \cdots & Y_r & 1 \\
\vdots & \vdots & \vdots & \ddots & \vdots & \vdots\\
Y_2-Z_{2,r-1} & Y_3-Z_{3,r-1} & Y_4-Z_{4,r-1} & \cdots & Y_r & 1\\
Y_2-Z_{2r} & Y_3-Z_{3r} & Y_4-Z_{4r} & \cdots & Y_r-Z_{rr} & 1
\end{pmatrix}$$
and
$$D=\diag(X_2^{-1},X_3^{-1},\ldots,X_r^{-1},1-(1-q)x_1).$$
\begin{lem}\label{lem3.4}
We have 
$$\begin{pmatrix}
\Phi^t_{0}(x_1,x_2,\ldots,x_{r+2};q)\\
\Phi^t_{1}(x_1,x_2,\ldots,x_{r+2};q)\\
\vdots\\
\Phi^t_{r-1}(x_1,x_2,\ldots,x_{r+2};q)
\end{pmatrix}=WD
\begin{pmatrix}
\Psi^t_{0}(u_1,u_2,\ldots,u_{r+2})-\Psi^t_{1}(u_1,u_2,\ldots,u_{r+2})\\
\Psi^t_{1}(u_1,u_2,\ldots,u_{r+2})-\Psi^t_{2}(u_1,u_2,\ldots,u_{r+2})\\
\vdots\\
\Psi^t_{r-2}(u_1,u_2,\ldots,u_{r+2})-\Psi^t_{r-1}(u_1,u_2,\ldots,u_{r+2})\\
\Psi^t_{r-1}(u_1,u_2,\ldots,u_{r+2})
\end{pmatrix},$$ where, for $j=2,\ldots,r+2$, 
\begin{align}\label{eqlem3.4}
\left\{
\begin{array}{l}
u_1=\frac{x_1}{1-(1-q)x_1},\\
u_j=\frac{x_{r+2}}{x_1^{r+2-j}}X_{j-1} \\
\hspace{1.2em}=\frac{1}{x_1^{r+2-j}}\left\{\sum\limits_{i=j-1}^{r}\left({i-1\atop j-2}\right)((1-q)x_1)^{i-j+1}
\left(x_1^{r+1-i}x_{i+1}-x_{r+2}\right)
+\frac{x_{r+2}}{\left(1-(1-q)x_1\right)^{j-1}}\right\}. 
\end{array}\right.\end{align}
\end{lem}
\begin{proof}
Let $j$ be an integer with $0\leq j\leq r-1$. 

\noindent {\bf Step 1.} By definitions and Lemma \ref{lem1}, we have
\begin{align*}
\Phi^t_{j}(q)=&\sum\limits_{k,l,h_1,\ldots,h_r\geq 0}G^t_{j}(k,l,h_1,\ldots,h_r;q)x_1^{k-l-\sum_{i=1}^rh_i}
x_2^{l-h_1}x_3^{h_1-h_2}\cdots x_{r+1}^{h_{r-1}-h_r}x_{r+2}^{h_r}\\
=&\sum\limits_{k,l,h_1,\ldots,h_r\geq 0}\sum\limits_{(k_1,\ldots,k_l)\in I_{j}(k,l,h_1,\ldots,h_r)}
\sum\limits_{a_1=2}^{k_1}\sum\limits_{a_2=1}^{k_2}\cdots\sum\limits_{a_l=1}^{k_l}
\left({k_1-2\atop a_1-2}\right)\left\{\prod_{i=2}^l\left({k_i-1\atop a_i-1}\right)\right\}\\
& \times (1-q)^{\sum\limits_{i=1}^l(k_i-a_i)}
\zeta^t_q(a_1,\ldots,a_l)x_1^{k-l-\sum_{i=1}^rh_i} x_2^{l-h_1}x_3^{h_1-h_2}\cdots x_{r+1}^{h_{r-1}-h_r}x_{r+2}^{h_r}.
\end{align*}
By changing the order of the sums, we get
\begin{align*}
\Phi^t_{j}(q)=&\sum\limits_{k,l,h_1,\ldots,h_r\geq 0}\sum\limits_{(a_1,\ldots,a_l)\in I_{0}(k,l,h_1,\ldots,h_r)}
\;\sum_{k_p'\geq a_p \atop{ p=1,\ldots, l \atop{k_1'\geq j+2}}}
\left({k_1'-2\atop a_1-2}\right)\left\{\prod_{i=2}^l \left({k_i'-1\atop a_i-1}\right)\right\}\\
&\times (1-q)^{\sum_{i=1}^l (k_i'-a_i)}x_1^{k'-l-\sum_{i=1}^r h_i'} x_2^{l-h_1'}x_3^{h_1'-h_2'}
\cdots x_{r+1}^{h_{r-1}'-h_r'}x_{r+2}^{h_r'}\zeta^t_q(a_1,\ldots,a_l),
\end{align*}
where $k',h_1',\ldots,h_r'$ are non-negative integers such that $(k_1',\ldots,k_l')\in I_{j}(k',l,$ $h_1',\ldots,h_r')$.
Since
$$h_1'+\cdots+h_r'=rl-r(l-h_1')-(r-1)(h_1'-h_2')-\cdots-(h_{r-1}'-h_r')$$
and
$$h_r'=l-(l-h_1')-(h_1'-h_2')-\cdots-(h_{r-1}'-h_r'),$$
we get
\begin{align*}
&x_1^{k'-l-\sum_{i=1}^r h_i'} x_2^{l-h_1'}x_3^{h_1'-h_2'}\cdots x_{r+1}^{h_{r-1}'-h_r'}x_{r+2}^{h_r'}\\
&=x_1^{k'-(r+1)l}x_{r+2}^l \left(\frac{x_1^rx_2}{x_{r+2}}\right)^{l-h_1'}
\left(\frac{x_1^{r-1}x_3}{x_{r+2}}\right)^{h_1'-h_2'}
\cdots\left(\frac{x_1x_{r+1}}{x_{r+2}}\right)^{h_{r-1}'-h_r'}.
\end{align*}
By using the facts that $l-h_1'=\sum_{i=1}^l \delta_{k_i',1}$, $h_{j-1}'-h_j'=\sum_{i=1}^l\delta_{k_i',j}$, $j=2,\ldots,r$
and $k'=k_1'+\cdots+k_l'$, we have
\begin{align*}
&x_1^{k'-l-\sum_{i=1}^rh_i'} x_2^{l-h_1'}x_3^{h_1'-h_2'}\cdots x_{r+1}^{h_{r-1}'-h_r'}x_{r+2}^{h_r'}\\
&=\left(\frac{x_{r+2}}{x_1^{r+1}}\right)^l \prod\limits_{i=1}^lx_1^{k_j'}\left(\frac{x_1^rx_2}{x_{r+2}}\right)^{\delta_{k_i',1}}
\left(\frac{x_1^{r-1}x_3}{x_{r+2}}\right)^{\delta_{k_i',2}}
\cdots\left(\frac{x_1x_{r+1}}{x_{r+2}}\right)^{\delta_{k_i',r}}.
\end{align*}
Hence
\begin{align*}
\Phi^t_{j}(q)=&\sum\limits_{k,l,h_1,\ldots,h_r\geq 0}\sum\limits_{(a_1,\ldots,a_l)\in I_{0}(k,l,h_1,\ldots,h_r)}
\;\sum_{k_p'\geq a_p \atop{ p=1,\ldots,l \atop{k_1'\geq j+2}}}\left({k_1'-2\atop a_1-2}\right)\\
&\times \left\{\prod_{i=2}^l\left({k_i'-1\atop a_i-1}\right)\right\}(1-q)^{\sum_{i=1}^l(k_i'-a_i)}\left(\frac{x_{r+2}}{x_1^{r+1}}\right)^l\\
&\times \prod\limits_{i=1}^l x_1^{k_i'}\left(\frac{x_1^rx_2}{x_{r+2}}\right)^{\delta_{k_i',1}}
\left(\frac{x_1^{r-1}x_3}{x_{r+2}}\right)^{\delta_{k_i',2}}
\cdots\left(\frac{x_1x_{r+1}}{x_{r+2}}\right)^{\delta_{k_i',r}}\zeta^t_q(a_1,\ldots,a_l)\\
=&\sum\limits_{k,l,h_1,\ldots,h_r\geq 0}\sum\limits_{(a_1,\ldots,a_l)\in I_{0}(k,l,h_1,\ldots,h_r)}
\;\sum_{k_p'\geq a_p \atop{ p=1,\ldots,l \atop{k_1'\geq j+2}}}\left({k_1'-2\atop a_1-2}\right)\\
&\times \left\{\prod_{i=2}^l \left({k_i'-1\atop a_i-1}\right)\right\}((1-q)x_1)^{\sum_{i=1}^l (k_i'-a_i)} \prod\limits_{i=1}^l
\left(\frac{x_1^rx_2}{x_{r+2}}\right)^{\delta_{k_i',1}}
\left(\frac{x_1^{r-1}x_3}{x_{r+2}}\right)^{\delta_{k_i',2}}\\
&\times \cdots \times \left(\frac{x_1x_{r+1}}{x_{r+2}}\right)^{\delta_{k_i',r}}x_1^k\left(\frac{x_{r+2}}{x_1^{r+1}}\right)^l\zeta^t_q(a_1,\ldots,a_l).
\end{align*}

\noindent {\bf Step 2.} For any fixed $(a_1,\ldots,a_l)\in I_{0}(k,l,h_1,\ldots,h_r)$, we set
\begin{align*}
S^{(j)}(a_1,\ldots,a_l)& =\sum_{k_p'\geq a_p \atop{ p=1,\ldots,l \atop{k_1'\geq j+2}}}\left({k_1'-2\atop a_1-2}\right)
\left\{\prod_{i=2}^l \left({k_i'-1\atop a_i-1}\right)\right\}((1-q)x_1)^{\sum_{i=1}^l (k_i'-a_i)} \\
& \times \prod\limits_{i=1}^l
\left(\frac{x_1^rx_2}{x_{r+2}}\right)^{\delta_{k_i',1}}
\left(\frac{x_1^{r-1}x_3}{x_{r+2}}\right)^{\delta_{k_i',2}} \cdots \left(\frac{x_1x_{r+1}}{x_{r+2}}\right)^{\delta_{k_i',r}}.
\end{align*}
Then $S^{(j)}(a_1,\ldots,a_l)=S_1^{(j)}S_2\cdots S_l$, where
$$S_1^{(j)}=\sum_{i\geq a_1,i\geq j+2}\left({i-2\atop a_1-2}\right)((1-q)x_1)^{i-a_1}
\left(\frac{x_1^rx_2}{x_{r+2}}\right)^{\delta_{i,1}}
\left(\frac{x_1^{r-1}x_3}{x_{r+2}}\right)^{\delta_{i,2}}
\cdots\left(\frac{x_1x_{r+1}}{x_{r+2}}\right)^{\delta_{i,r}}$$
and
$$S_m=\sum_{i\geq a_m}\left({i-1\atop a_m-1}\right)((1-q)x_1)^{i-a_m}
\left(\frac{x_1^rx_2}{x_{r+2}}\right)^{\delta_{i,1}}
\left(\frac{x_1^{r-1}x_3}{x_{r+2}}\right)^{\delta_{i,2}}
\cdots\left(\frac{x_1x_{r+1}}{x_{r+2}}\right)^{\delta_{i,r}}$$ for
$m=2,\ldots,l$.

To compute $S_1^{(j)}$, we consider the following three cases.

\noindent  {\bf Case (11):}  $a_1\geq r+1$.

In this case, we have
$$S_1^{(j)}=\sum_{i\geq a_1}\left({i-2\atop a_1-2}\right)((1-q)x_1)^{i-a_1}=(1-(1-q)x_1)^{1-a_1}.$$

\noindent {\bf Case (12):} $a_1=p$ for some $p$ with $j+2\leq p\leq r$.

In this case, we have
\begin{align*}
S_1^{(j)}&=\sum_{i\geq p}\left({i-2\atop p-2}\right)((1-q)x_1)^{i-p}
\left(\frac{x_1^rx_2}{x_{r+2}}\right)^{\delta_{i,1}}
\left(\frac{x_1^{r-1}x_3}{x_{r+2}}\right)^{\delta_{i,2}}
\cdots\left(\frac{x_1x_{r+1}}{x_{r+2}}\right)^{\delta_{i,r}}\\
&=\sum_{i=p}^r\left({i-2\atop p-2}\right)((1-q)x_1)^{i-p}
\frac{x_1^{r+1-i}x_{i+1}}{x_{r+2}}+\sum_{i\geq r+1}\left({i-2\atop p-2}\right)((1-q)x_1)^{i-p} \\
&=\sum_{i=p}^r\left({i-2\atop p-2}\right)((1-q)x_1)^{i-p}
\frac{x_1^{r+1-i}x_{i+1}-x_{r+2}}{x_{r+2}}+(1-(1-q)x_1)^{1-p}\\
&=Y_p.
\end{align*}

\noindent {\bf Case (13):} $a_1=p$ for some $p$ with $2\leq p\leq j+1$.

In this case,
\begin{align*}
S_1^{(j)}&=\sum_{i\geq j+2}\left({i-2\atop
p-2}\right)((1-q)x_1)^{i-p}
\left(\frac{x_1^rx_2}{x_{r+2}}\right)^{\delta_{i,1}}
\left(\frac{x_1^{r-1}x_3}{x_{r+2}}\right)^{\delta_{i,2}}
\cdots\left(\frac{x_1x_{r+1}}{x_{r+2}}\right)^{\delta_{i,r}}\\
&=\sum_{i=j+2}^r\left({i-2\atop p-2}\right)((1-q)x_1)^{i-p}
\frac{x_1^{r+1-i}x_{i+1}}{x_{r+2}}+\sum_{i\geq r+1}\left({i-2\atop
p-2}\right)((1-q)x_1)^{i-p}\\
&=Y_p-Z_{p,j+1}.
\end{align*}

Now 
we compute $S_m$ for $2\leq m\leq l$. We have two
cases.

\noindent {\bf Case (m1):} $a_m=p$ for some $1\leq p\leq r$.

Then we have
\begin{align*}
S_m&=\sum_{i=p}^r\left({i-1\atop p-1}\right)((1-q)x_1)^{i-p}
\frac{x_1^{r+1-i}x_{i+1}}{x_{r+2}}+\sum_{i\geq r+1}\left({i-1\atop p-1}\right)((1-q)x_1)^{i-p}\\
&=\sum_{i=p}^r\left({i-1\atop p-1}\right)((1-q)x_1)^{i-p}
\frac{x_1^{r+1-i}x_{i+1}-x_{r+2}}{x_{r+2}}+(1-(1-q)x_1)^{-p}\\
&=X_p.
\end{align*}

\noindent {\bf Case (m2):} $a_m\geq r+1$.

Then
$$S_m=\sum_{i\geq a_m}\left({i-1\atop a_m-1}\right)((1-q)x_1)^{i-a_m}=(1-(1-q)x_1)^{-a_m}.$$

Setting $\widetilde{X}_p=X_p(1-(1-q)x_1)^p$ for $p=1,\ldots,r$, we obtain $S^{(j)}(a_1,\ldots,a_l)$. There are three cases.

\noindent {\bf Case (i):} $a_1\geq r+1$.

In this case,
\begin{align*}
S^{(j)}(a_1,\ldots,a_l)=&(1-(1-q)x_1)^{1-a_1}X_1^{l-h_1}X_2^{h_1-h_2}\cdots X_r^{h_{r-1}-h_r}\\
&\times \frac{(1-(1-q)x_1)^{-\sum_{m=2}^l a_m}}{(1-(1-q)x_1)^{-\{(l-h_1)+2(h_1-h_2)+\cdots+r(h_{r-1}-h_r)\}}}\\
=&\frac{1-(1-q)x_1}{(1-(1-q)x_1)^k}\left\{X_1(1-(1-q)x_1)\right\}^{l-h_1}\\
&\times \left\{X_2(1-(1-q)x_1)^2\right\}^{h_1-h_2}\cdots
\left\{X_r(1-(1-q)x_1)^r\right\}^{h_{r-1}-h_r}\\
 =&\frac{1-(1-q)x_1}{(1-(1-q)x_1)^k}\widetilde{X}_1^{l-h_1}
\widetilde{X}_2^{h_1-h_2}\cdots \widetilde{X}_r^{h_{r-1}-h_r}.
\end{align*}

\noindent {\bf Case (ii):} $a_1=p$ for some $p$ with $j+2\leq p\leq r$.

Then
\begin{align*}
&S^{(j)}(a_1,\ldots,a_l)\\
& =Y_p X_1^{l-h_1}X_2^{h_1-h_2}\cdots 
X_{p-1}^{h_{p-2}-h_{p-1}} X_p^{h_{p-1}-h_{p}-1} X_{p+1}^{h_{p}-h_{p+1}}\cdots
X_r^{h_{r-1}-h_r}\\
&\times \frac{(1-(1-q)x_1)^{-\sum_{m=2}^l a_m}}{(1-(1-q)x_1)^{-\{(l-h_1)+2(h_1-h_2)+\cdots+ (p-1)(h_{p-2}-h_{p-1})+ p(h_{p-1}-h_p-1)+ (p+1)(h_{p}-h_{p+1})+ \cs +r(h_{r-1}-h_r)\}}}\\
&=\frac{Y_pX_p^{-1}}{(1-(1-q)x_1)^k}\widetilde{X}_1^{l-h_1}
\widetilde{X}_2^{h_1-h_2}\cdots \widetilde{X}_r^{h_{r-1}-h_r}.
\end{align*}

\noindent {\bf Case (iii):} $a_1=p$ for some $p$ with $2\leq p\leq j+1$.

Then
$$S^{(j)}(a_1,\ldots,a_l)=\frac{(Y_p-Z_{p,j+1})X_p^{-1}}{(1-(1-q)x_1)^k}\widetilde{X}_1^{l-h_1}
\widetilde{X}_2^{h_1-h_2}\cdots \widetilde{X}_r^{h_{r-1}-h_r}.$$

\noindent {\bf Step 3.} Notice that $I_0(k,l,h_1,\ldots,h_r)$ equates to the disjoint union
$$I_{r-1}(k,l,h_1,\ldots,h_r)\sqcup \bigsqcup_{i=0}^{r-2}\left(I_{i}(k,l,h_1,\ldots,h_r)
\setminus I_{i+1}(k,l,h_1,\ldots,h_r)\right).$$ 
Therefore
\begin{align*}
\Phi^t_{j}(q)
= & \sum\limits_{k,l,h_1,\ldots,h_r\geq
0}\left\{\sum\limits_{(a_1,\ldots,a_l)\in
I_{r-1}(k,l,h_1,\ldots,h_r)}+\sum\limits_{i=2}^{r}\sum\limits_{(a_1,\ldots,a_l)\in
I_{i-2}(k,l,h_1,\ldots,h_r) \atop {a_1=i}}\right\}\\
&\times S^{(j)}(a_1,\ldots,a_l)x_1^k\left(\frac{x_{r+2}}{x_1^{r+1}}\right)^l
\zeta^t_q(a_1,\ldots,a_l).
\end{align*}

Using the results of $S^{(j)}(a_1,\ldots,a_n)$ obtained in Step 2, we have
\begin{align*}
\Phi^t_j(q)=&\sum\limits_{k,l,h_1,\ldots,h_r\geq
0}\sum\limits_{(a_1,\ldots,a_l)\in I_{r-1}(k,l,h_1,\ldots,h_r)}
\frac{1-(1-q)x_1}{(1-(1-q)x_1)^k}\widetilde{X}_1^{l-h_1}
\widetilde{X}_2^{h_1-h_2}\cdots
\widetilde{X}_r^{h_{r-1}-h_r}\\
&\times
x_1^k\left(\frac{x_{r+2}}{x_1^{r+1}}\right)^l \zeta^t_q(a_1,\ldots,a_l)
+\sum\limits_{i=2}^{j+1}\sum\limits_{k,l,h_1,\ldots,h_r\geq
0}\sum\limits_{(a_1,\ldots,a_l)\in
I_{i-2}(k,l,h_1,\ldots,h_r) \atop{a_1=i}}\\
&\times
\frac{(Y_i-Z_{i,j+1})X_i^{-1}}{(1-(1-q)x_1)^k}\widetilde{X}_1^{l-h_1}
\widetilde{X}_2^{h_1-h_2}\cdots
\widetilde{X}_r^{h_{r-1}-h_r}x_1^k\left(\frac{x_{r+2}}{x_1^{r+1}}\right)^l\zeta^t_q(a_1,\ldots,a_l)\\
&+\sum\limits_{i=j+2}^{r}\sum\limits_{k,l,h_1,\ldots,h_r\geq
0}\sum\limits_{(a_1,\ldots,a_l)\in
I_{i-2}(k,l,h_1,\ldots,h_r) \atop{a_1=i}}\frac{Y_iX_i^{-1}}{(1-(1-q)x_1)^k}\widetilde{X}_1^{l-h_1}
\widetilde{X}_2^{h_1-h_2}\\
&\times \cdots \times
\widetilde{X}_r^{h_{r-1}-h_r}x_1^k\left(\frac{x_{r+2}}{x_1^{r+1}}\right)^l\zeta^t_q(a_1,\ldots,a_l)\\
=&(1-(1-q)x_1)\sum\limits_{k,l,h_1,\ldots,h_r\geq
0}\xi^t_{r-1}(k,l,h_1,\ldots,h_r)\left(\frac{x_1}{1-(1-q)x_1}\right)^k
\left(\frac{x_{r+2}}{x_1^{r+1}}\right)^l\\
&\times \widetilde{X}_1^{l-h_1} \widetilde{X}_2^{h_1-h_2}\cdots
\widetilde{X}_r^{h_{r-1}-h_r}+\sum\limits_{i=2}^{j+1}\frac{Y_i-Z_{i,j+1}}{X_i}
\sum\limits_{k,l,h_1,\ldots,h_r\geq
0}\left\{\xi^t_{i-2}(k,l,h_1,\ldots,h_r)\right. \\
&\left. -\xi^t_{i-1}(k,l,h_1,\ldots,h_r)\right\}\left(\frac{x_1}{1-(1-q)x_1}\right)^k
\left(\frac{x_{r+2}}{x_1^{r+1}}\right)^l\widetilde{X}_1^{l-h_1}
\widetilde{X}_2^{h_1-h_2}\cdots \widetilde{X}_r^{h_{r-1}-h_r}\\
&+\sum\limits_{i=j+2}^{r}\frac{Y_i}{X_i}\sum\limits_{k,l,h_1,\ldots,h_r\geq
0}\left\{ \xi^t_{i-2}(k,l,h_1,\ldots,h_r)-\xi^t_{i-1}(k,l,h_1,\ldots,h_r)\right\}\\
&\times \left(\frac{x_1}{1-(1-q)x_1}\right)^k
\left(\frac{x_{r+2}}{x_1^{r+1}}\right)^l\widetilde{X}_1^{l-h_1}
\widetilde{X}_2^{h_1-h_2}\cdots \widetilde{X}_r^{h_{r-1}-h_r}.
\end{align*}

\noindent {\bf Step 4.} Since
$$l=(l-h_1)+(h_1-h_2)+\cdots+(h_{r-1}-h_r)+h_r$$
and
$$l+h_1+h_2+\cdots+h_r=(l-h_1)+2(h_1-h_2)+3(h_2-h_3)+
\cdots+r(h_{r-1}-h_r)+(r+1)h_r,$$
and \eqref{eqlem3.4}, we have
\begin{align*}
&\left(\frac{x_1}{1-(1-q)x_1}\right)^k
\left(\frac{x_{r+2}}{x_1^{r+1}}\right)^l \widetilde{X}_1^{l-h_1}
\widetilde{X}_2^{h_1-h_2}\cdots \widetilde{X}_r^{h_{r-1}-h_r}\\
=&u_1^{k-l-\sum_{i=1}^l h_i}
\left(\frac{u_1x_{r+2}\widetilde{X}_1}{x_1^{r+1}}\right)^{l-h_1}
\left(\frac{u_1^2x_{r+2}\widetilde{X}_2}{x_1^{r+1}}\right)^{h_1-h_2}
\cdots\left(\frac{u_1^rx_{r+2}\widetilde{X}_r}{x_1^{r+1}}\right)^{h_{r-1}-h_r}
\left(\frac{u_1^{r+1}x_{r+2}}{x_1^{r+1}}\right)^{h_r}\\
=&u_1^{k-l-\sum_{i=1}^l h_i} u_2^{l-h_1}u_{3}^{h_1-h_2}\cdots
u_{r+1}^{h_{r-1}-h_r}u_{r+2}^{h_r}.
\end{align*}
Thus
\begin{align*}
\Phi^t_{j}(q)=&(1-(1-q)x_1)\Psi^t_{r-1}(u_1,\ldots,u_{r+2})+\sum\limits_{i=2}^{j+1}\frac{Y_i-Z_{i,j+1}}{X_i}
\left\{ \Psi^t_{i-2}(u_1,\ldots,u_{r+2})\right.\\
&\left. -\Psi^t_{i-1}(u_1,\ldots,u_{r+2})\right\}
+\sum\limits_{i=j+2}^r\frac{Y_i}{X_i}
\left\{ \Psi^t_{i-2}(u_1,\ldots,u_{r+2})-\Psi^t_{i-1}(u_1,\ldots,u_{r+2})\right\}.
\end{align*}
Hence we complete the proof of Lemma \ref{lem3.4}.\end{proof}
\begin{rmk}
We find that \eqref{eqlem3.4} is equivalent to \eqref{x1} and \eqref{xj} as shown in the proof of Lemma \ref{prop3.6} below. 
\end{rmk}
\begin{lem}\label{lem3.5}
Let $T=(t_{ij})$ be the $n\times n$ lower triangular matrix defined
by
\begin{align*}
t_{ij}= 
\left\{ \begin{array}{ll} {i-1 \choose j-1} x^{i-j}, & i \geq j, \\
 0, & i < j. \end{array} \right. 
\end{align*}
then the invertible matrix of $T$ is the lower triangular matrix
$(t_{ij}')$ defined by
\begin{align*}
t_{ij}'= 
\left\{ \begin{array}{ll} {i-1 \choose j-1} (-x)^{i-j}, & i \geq j, \\
 0, &  i < j. \end{array} \right. 
 \end{align*}
\end{lem}
\begin{proof}
This is easily shown by checking that $T(t_{ij}')$ is the identity matrix. 
%
\end{proof}
\begin{lem}\label{prop3.6}
Let $M$ be a $r\times r$ matrix given by 
$$M=\begin{pmatrix}
X_2-\frac{Y_2}{1-(1-q)x_1} & X_3-\frac{Y_3}{1-(1-q)x_1} &\cdots &
X_{r-1}-\frac{Y_{r-1}}{1-(1-q)x_1} & X_r-\frac{Y_r}{1-(1-q)x_1} & \frac{1}{1-(1-q)x_1}\\
-\frac{Y_2}{1-(1-q)x_1} & X_3-\frac{Y_3}{1-(1-q)x_1} &\cdots &
X_{r-1}-\frac{Y_{r-1}}{1-(1-q)x_1} & X_r-\frac{Y_r}{1-(1-q)x_1} & \frac{1}{1-(1-q)x_1}\\
-\frac{Y_2}{1-(1-q)x_1} & -\frac{Y_3}{1-(1-q)x_1} &\cdots &
X_{r-1}-\frac{Y_{r-1}}{1-(1-q)x_1} & X_r-\frac{Y_r}{1-(1-q)x_1} & \frac{1}{1-(1-q)x_1}\\
\cdots & \cdots & \ddots & \ddots & \cdots & \cdots\\
-\frac{Y_2}{1-(1-q)x_1} & -\frac{Y_3}{1-(1-q)x_1} &\cdots &
-\frac{Y_{r-1}}{1-(1-q)x_1} & X_r-\frac{Y_r}{1-(1-q)x_1} & \frac{1}{1-(1-q)x_1}\\
-\frac{Y_2}{1-(1-q)x_1} & -\frac{Y_3}{1-(1-q)x_1} &\cdots &
-\frac{Y_{r-1}}{1-(1-q)x_1} & -\frac{Y_r}{1-(1-q)x_1} & \frac{1}{1-(1-q)x_1}\\
\end{pmatrix}. $$
Let $M_1$ 
be the $(r-1)\times (r-1)$ 
matrix defined in Proposition \ref{lem10}.   
Then we have 
$$\begin{pmatrix}
\Psi^t_0(u_1,\ldots,u_{r+2})\\
\Psi^t_1(u_1,\ldots,u_{r+2})\\
\vdots\\
\Psi^t_{r-1}(u_1,\ldots,u_{r+2})
\end{pmatrix}=M
\begin{pmatrix}
0 & M_1\\
1 & 0
\end{pmatrix}
\begin{pmatrix}
\Phi^t_0(x_1,\ldots,x_{r+2};q)\\
\Phi^t_0(x_1,\ldots,x_{r+2};q)-\Phi^t_1(x_1,\ldots,x_{r+2};q)\\
\Phi^t_0(x_1,\ldots,x_{r+2};q)-\Phi^t_2(x_1,\ldots,x_{r+2};q)\\
\vdots\\
\Phi^t_0(x_1,\ldots,x_{r+2};q)-\Phi^t_{r-1}(x_1,\ldots,x_{r+2};q)
\end{pmatrix}.$$
\end{lem}
\begin{proof}
The lemma follows from Lemma \ref{lem3.4}.

\noindent {\bf Step 1.} We first remark that from \eqref{eqlem3.4}, we can represent
$x_1,\ldots,x_{r+2}$ by $u_1,\ldots,u_{r+2}$.
In fact, by \eqref{eqlem3.4}, it is trivial that
$$x_1=\frac{u_1}{1+(1-q)u_1},\;x_{r+2}=\frac{u_{r+2}}{(1+(1-q)u_1)^{r+1}}.$$
We rewrite the formulas in \eqref{eqlem3.4} in matrix form
$$\begin{pmatrix}
u_2\\
u_3\\
\vdots\\
u_{r+1}
\end{pmatrix}=
T_1\begin{pmatrix}
x_2-\frac{x_{r+2}}{x_1^r}\\
x_3-\frac{x_{r+2}}{x_1^{r-1}}\\
\vdots\\
x_{r+1}-\frac{x_{r+2}}{x_1}
\end{pmatrix}+x_{r+2}
\begin{pmatrix}
\frac{1}{x_1^r(1-(1-q)x_1)}\\
\frac{1}{x_1^{r-1}(1-(1-q)x_1)^2}\\
\vdots\\
\frac{1}{x_1(1-(1-q)x_1)^r}
\end{pmatrix},$$
where $T_1=(t_{ij}^{(1)})$ is a $r\times r$ upper triangular matrix with
$$t_{ij}^{(1)}=\left\{
\begin{array}{ll}
\left({j-1\atop i-1}\right)(1-q)^{j-i}, & i\leq j,\\
0, & i>j.
\end{array}\right.$$
Hence
$$\begin{pmatrix}
x_2\\
x_3\\
\vdots\\
x_{r+1}
\end{pmatrix}=T_1^{-1}
\begin{pmatrix}
u_2-\frac{u_{r+2}}{u_1^r}\\
u_3-\frac{u_{r+2}}{u_1^{r-1}}\\
\vdots\\
u_{r+1}-\frac{u_{r+2}}{u_1}
\end{pmatrix}+\frac{u_{r+2}}{u_1^{r+1}}
\begin{pmatrix}
\frac{u_1}{1+(1-q)u_1}\\
\left(\frac{u_1}{1+(1-q)u_1}\right)^2\\
\vdots\\
\left(\frac{u_1}{1+(1-q)u_1}\right)^{r}
\end{pmatrix}.$$
We see that $T_1^{-1}$ is obtained from Lemma \ref{lem3.5}. A direct computation gives
the representations of $x_2,\ldots,x_{r+1}$ by $u_1,u_2,\ldots,u_{r+2}$ as we state at \eqref{x1} and \eqref{xj}.

\noindent {\bf Step 2.}
If we use the same notations as in Lemma \ref{lem3.4}, then
$$\begin{pmatrix}
\Psi^t_0(u_1,\ldots,u_{r+2})\\
\Psi^t_1(u_1,\ldots,u_{r+2})\\
\vdots\\
\Psi^t_{r-1}(u_1,\ldots,u_{r+2})
\end{pmatrix}=
\begin{pmatrix}
 1  & 1 & \cdots & 1 \\
   & 1 & \cdots & 1 \\
     &  & \ddots& \vdots\\
   &   &  &  1
\end{pmatrix}D^{-1}W^{-1}
\begin{pmatrix}
\Phi^t_0(x_1,\ldots,x_{r+2};q)\\
\Phi^t_1(x_1,\ldots,x_{r+2};q)\\
\vdots\\
\Phi^t_{r-1}(x_1,\ldots,x_{r+2};q)
\end{pmatrix}.$$
Now
$$\begin{pmatrix}
1\\
1 & -1 \\
1& & -1\\
\vdots&&& \ddots\\
1 &&&& -1
\end{pmatrix}W
\begin{pmatrix}
1\\
&1\\
&&\ddots\\
&&& 1\\
-Y_2 & -Y_3 & \cdots & -Y_r & 1
\end{pmatrix}=\begin{pmatrix}
0 & 1 \\
T_2& 0
\end{pmatrix},$$
where
$$T_2=\begin{pmatrix}
Z_{22} \\
Z_{23} & Z_{33}\\
Z_{24} & Z_{34} & Z_{44} \\
\vdots& && \ddots\\
Z_{2r} & Z_{3r} & Z_{4r} & \cdots & Z_{rr}
\end{pmatrix}$$
is a $(r-1)\times (r-1)$ matrix.
Hence
$$W^{-1}=
\begin{pmatrix}
1\\
&1\\
&&\ddots\\
&&& 1\\
-Y_2 & -Y_3 & \cdots & -Y_r & 1
\end{pmatrix}
\begin{pmatrix}
0 & T_2^{-1}\\
1 & 0
\end{pmatrix}
\begin{pmatrix}
1\\
1 & -1 \\
1& & -1\\
\vdots&&& \ddots\\
1 &&&& -1
\end{pmatrix}.
$$
On the other hand,  direct computations give that
$$M=\begin{pmatrix}
 1  & 1 & \cdots & 1 \\
   & 1 & \cdots & 1 \\
     &  & \ddots\\
   &   &  &  1
\end{pmatrix}D^{-1}
\begin{pmatrix}
1\\
&1\\
&&\ddots\\
&&& 1\\
-Y_2 & -Y_3 & \cdots & -Y_r & 1
\end{pmatrix}. $$
Thus we only need to show that $T_2^{-1}=M_1$.

\noindent {\bf Step 3.} By definition, 
$$Z_{p,j+1}-Z_{pj}=\frac{1}{x_{r+2}}\left({j-1\atop p-2}\right)(1-q)^{j+1-p}x_{1}^{r+1-p}x_{j+2}$$
for $2\leq p \leq j\leq r-1$ and
$$Z_{pp}=\frac{1}{x_{r+2}}x_1^{r+1-p}x_{p+1}$$
for $2\leq p\leq r$. Hence we get
$$\begin{pmatrix}
1 \\
-1 & 1 \\
& -1 & 1\\
&& \ddots & \ddots\\
&&&-1 & 1
\end{pmatrix}T_2=
\frac{1}{x_{r+2}}
\begin{pmatrix}
x_3 \\
 & x_4 \\
 && \ddots\\
 &&& x_{r+1}
\end{pmatrix}T_3
\begin{pmatrix}
x_1^{r-1}\\
& x_1^{r-2}\\
&&\ddots\\
&&&x_1
\end{pmatrix},$$
where $T_3=(t_{ij}^{(3)})$ is the lower triangular matrix defined by
$$t_{ij}^{(3)}=\left\{\begin{array}{ll}
\left({i-1\atop j-1}\right)(1-q)^{i-j}, & i\geq j,\\
0, & i<j.
\end{array}
\right.$$
From Lemma \ref{lem3.5}, $T_3^{-1}=T$ (which is given in Proposition \ref{lem10}), thus we obtain this lemma.
\end{proof}
\begin{proof}[Proof of Proposition \ref{lem10}]
This is nothing but the first component of the identity in Lemma \ref{prop3.6}. 
\end{proof}
\subsection{Relation between $\Psi^t_0$ and ${}_{r+2} \phi_{r+1}$}
We recall that $\widetilde{A^{(i)}_j}$ is given in Theorem \ref{thm9} and 
$c, \widetilde{B_j}, v_j$ are in Corollary \ref{Cor}. Under these notations, we have the following proposition. 
\begin{prop}\label{Prop} 
We have 
\begin{align*}
\Psi^t_0
& = \f{1}{1-(1-q)x_1} \f{1}{c} \sum_{j=0}^{r-1} \widetilde{A^{(0)}_j} \widetilde{B_j} v_j \\
& \quad + \f{x_{r+2}}{c} \sum_{j=0}^{r-1} \left\{ \sum_{i=2}^{r} \f{1}{x_1^{r+1-i}} \left( X_{i} -\f{Y_{i}}{1-(1-q)x_1} \right)c_{i-1, j+1} \right\}\widetilde{B_j} v_j \\
 & = \f{1}{c(1-(1-q)x_1)} \sum_{j=0}^{r-1} A_j \widetilde{B_j} v_j
\end{align*}
where 
$$A_j = \widetilde{A^{(0)}_j} + x_{r+2} \sum_{i=2}^{r} \left\{ X_i \left( 1-(1-q)x_1 \right) -Y_i \right\} \f{c_{i-1, j+1}}{x_1^{r+1-i}}$$
with $0 \leq j \leq r-1$ and 
$c_{i,j}$'s are given in Lemma \ref{lem11} below. 
\end{prop}
\noindent
Before we prove this proposition,  we prepare two lemmas. 
\begin{lem}\label{Lem22}We have 
\begin{align*}\nonumber
& \underline{X} M_1\begin{pmatrix} \Phi^t_0(q) - \Phi^t_1(q)\\ \Phi^t_0(q) - \Phi^t_2(q)\\ \vdots\\ \Phi^t_0(q) - \Phi^t_{r-1}(q) \end{pmatrix}\\ \nonumber
& = \f{x_{r+2}}{c} \underline{X} \begin{pmatrix} x_1^{-(r-1)} & & & \\ & x_1^{-(r-2)} & & \\ & & \ddots & \\ & & & x_1^{-1} \end{pmatrix}
T 
\begin{pmatrix} x_3^{-1} & & & \\ & x_4^{-1} & & \\ & & \ddots & \\ & & & x_{r+1}^{-1} \end{pmatrix}\\ 
& \quad \times \begin{pmatrix} 1 & -1 & & & \\ &1 & -1 & & \\ & & \ddots& \ddots & \\ & & &1& -1 \end{pmatrix}
A 
\begin{pmatrix}
\widetilde{B_0} v_0 \\ \widetilde{B_1}v_1 \\  \vdots \\ \widetilde{B_{r-1}} v_{r-1}  
\end{pmatrix}, 
\end{align*} 
where $\underline{X}$, $T$ and $M_1$ are same as those in Proposition \ref{lem10} and $A$ is same as that in Corollary \ref{Cor}.  
\end{lem}
\begin{proof}
Since
\begin{align*}
\begin{pmatrix} 1 & & & \\ -1&1 & & \\ & \ddots& \ddots & \\ & &-1& 1 \end{pmatrix}
\begin{pmatrix} \Phi^t_0(q) - \Phi^t_1(q)\\ \Phi^t_0(q) - \Phi^t_2(q)\\ \vdots\\ \Phi^t_0(q) - \Phi^t_{r-1}(q) \end{pmatrix}
=\begin{pmatrix} 1 & -1 & & & \\ &1 & -1 & & \\ & & \ddots& \ddots & \\ & & &1& -1 \end{pmatrix}
\begin{pmatrix} \Phi^t_0(q) \\ \Phi^t_1(q) \\ \vdots\\ \Phi^t_{r-1}(q) \end{pmatrix}
\end{align*}
and Corollary \ref{Cor}, we have the lemma.  
\end{proof}
\begin{lem}\label{lem11}
We have the following identities.  
\begin{itemize}
\item[(i)] 
$$\widetilde{A^{(j)}_{r-1-j}} = x_{j+3} S_q(r-j, r-j)$$
for $j=0, 1, \ls, r-1$. 
\item[(ii)] 
$$\widetilde{A^{(j)}_{i}} - \widetilde{A^{(j+1)}_{i}}= x_{j+3} \left\{ S_q(r-j, i+1) -x_1 S_q(r-j-1, i+1) \right\}$$
for $j=0, 1, \ls, r-2$ and $i=0, 1, \ls, r-2-j$.  
\item[(iii)] 
$$
\begin{pmatrix} x_3^{-1} & & & \\ & x_4^{-1} & & \\ & & \ddots & \\ & & & x_{r+1}^{-1} \end{pmatrix}
\begin{pmatrix} 1 & -1 & & & \\ &1 & -1 & & \\ & & \ddots& \ddots & \\ & & &1& -1 \end{pmatrix}
A = (a_{ij})_{(r-1) \times r}$$
with
$$a_{ij} = \left\{ \begin{array}{ll}
S_q(r+1-i, j) - x_1 S_q(r-i, j), & i+j \leq r+1, \\
0, & \text{otherwise},  
\end{array}\right. $$
where $A$ is given in Corollary \ref{Cor}. 
\item[(iv)]
$$T \begin{pmatrix} x_3^{-1} & & & \\ & x_4^{-1} & & \\ & & \ddots & \\ & & & x_{r+1}^{-1} \end{pmatrix}
\begin{pmatrix} 1 & -1 & & & \\ &1 & -1 & & \\ & & \ddots& \ddots & \\ & & &1& -1 \end{pmatrix}
A = (c_{ij})_{(r-1) \times r}$$
with 
$$c_{ij} = \sum_{m=1}^{min\left\{ i, r+1-j \right\}}(q-1)^{i-m}{i-1 \choose m-1} a_{mj}. $$
\end{itemize}
\end{lem}
\begin{proof}
%
(i) Considering the case of $i=r-1-j$ in the definition $\widetilde{A^{(j)}_{i}}$ (in Theorem \ref{thm9}), we have
\begin{align*}
\widetilde{A^{(j)}_{r-1-j}} & = (x_{j+3} - x_1 x_{j+2}) S_q(r-j, r-j) + x_1 x_{j+2} S_q(r-j, r-j) \\
& = x_{j+3} S_q(r-j, r-j). 
\end{align*}
(ii) By definition we get
\begin{align*}
& \widetilde{A^{(j)}_{i}} - \widetilde{A^{(j+1)}_{i}} \\
& = \sum_{m=i}^{r-1-j} (x_{r+2-m} - x_1 x_{r+1-m}) S_q(m+1, i+1) +x_1 x_{j+2}S_q(r-j, i+1)\\
&\quad  -\sum_{m=i}^{r-2-j} (x_{r+2-m} - x_1 x_{r+1-m}) S_q(m+1, i+1)
-x_1 x_{j+3} S_q(r-1-j, i+1)\\
& = (x_{j+3} - x_1 x_{j+2}) S_q(r-j, i+1) + x_1 x_{j+2} S_q(r-j, i+1) \\
& \quad - x_1x_{j+3} S_q(r-1-j, i+1)\\
& = x_{j+3} \left( S_q(r-j, i+1) - x_1 S_q(r-1-j, i+1) \right). 
\end{align*}
(iii) By using (i) and (ii), we have
\begin{align*}
\begin{pmatrix} 1 & -1 & & & \\ &1 & -1 & & \\ & & \ddots& \ddots & \\ & & &1& -1 \end{pmatrix} A & = \begin{pmatrix} \widetilde{A^{(0)}_0} - \widetilde{A^{(1)}_0} & \widetilde{A^{(0)}_1} - \widetilde{A^{(1)}_1} & \cs & \widetilde{A^{(0)}_{r-2}} - \widetilde{A^{(1)}_{r-2}} & \widetilde{A^{(0)}_{r-1}}  \\
 \widetilde{A^{(1)}_0} - \widetilde{A^{(2)}_0} & \widetilde{A^{(1)}_1} - \widetilde{A^{(2)}_1} & \cs & \widetilde{A^{(1)}_{r-2}} & \\
 \vdots & \vdots & \iddots & & \\
 \widetilde{A^{(r-2)}_0} - \widetilde{A^{(r-1)}_0} & \widetilde{A^{(r-2)}_1}& &  & 
 \end{pmatrix}\\
 & = 
 \begin{pmatrix} x_3 a_{11} & x_3 a_{12} & \cs & x_3 a_{1, r-1} & x_3 a_{1r} \\
 x_4 a_{21} & x_4 a_{22} & \cs & x_4 a_{2,{r-1}} & \\
 \vdots & \vdots & \iddots & & \\
 x_{r+1} a_{r-1, 1} & x_{r+1} a_{r-1, 2} & & & 
 \end{pmatrix}. 
\end{align*}
(iv) 
By using (iii), we have 
\begin{align*}
c_{ij} & = \sum_{m=1}^{r-1} t_{im} a_{mj} 
 = \sum_{{{{1 \leq m \leq r-1} \atop {i \geq m}} \atop {m+j \leq r+1}}} t_{im} a_{mj} 
 = \sum_{m=1}^{min\left\{ i, r+1-j \right\}} { i-1 \choose m-1} (q-1)^{i-m} a_{mj}. 
\end{align*}
\end{proof}
\begin{proof}[Proof of Proposition \ref{Prop}]
Using Lemma \ref{Lem22} and Lemma \ref{lem11} (iv), we have
\begin{align*}
\underline{X} M_1 \begin{pmatrix} \Phi^t_0(q) - \Phi^t_1(q)\\ \Phi^t_0(q) - \Phi^t_2(q)\\ \vdots\\ \Phi^t_0(q) - \Phi^t_{r-1}(q) \end{pmatrix} 
 = \f{x_{r+2}}{c} \underline{X} \begin{pmatrix} x_1^{-(r-1)} & & & \\ & x_1^{-(r-2)} & & \\ & & \ddots & \\ & & & x_1^{-1} \end{pmatrix}
(c_{ij})
\begin{pmatrix}
\widetilde{B_0} v_0 \\ \widetilde{B_1} v_1 \\  \vdots \\ \widetilde{B_{r-1}} v_{r-1}  
\end{pmatrix}.
\end{align*}
The right-hand side is equal to 
$$\f{x_{r+2}}{c} \sum_{j=0}^{r-1} 
\left\{ \sum_{i=2}^{r} \f{1}{x_1^{r+1-i}} \left( X_{i} -\f{Y_{i}}{1-(1-q)x_1} \right)c_{i-1, j+1} \right\}\widetilde{B_j} v_j. $$
Applying this and Corollary \ref{Cor} to Proposition \ref{lem10}, we conclude Proposition \ref{Prop}. 
\end{proof}
\subsection{Expression of $A_j$ in terms of $q$-Stirling numbers}
Put $q_1 = 1-q$. We have the following expression of $A_j$ in Proposition \ref{Prop}. 

\begin{prop}\label{propAj}
For $0 \leq j \leq r-1$, we have
$$A_j = \sum_{m=j}^{r-1} c_m S_q(m+1, j+1) + \f{u_1 u_2}{(1+q_1u_1)^2} S_q(r, j+1), $$
where
\begin{align*}
{c}_m = & \sum\limits_{k=r-m+1}^{r+1}\left\{\left({k-2\atop
r-m}\right)+\frac{q_1u_1}{1+q_1u_1}\left({k-2\atop
r-m-1}\right)\right\}(-q_1)^{k-r+m-2}\\
& \times \left(\frac{u_k}{1+q_1u_1}-u_1^{k-r-2}u_{r+2}
+\frac{q_1u_{k+1}}{1+q_1u_1}\right). 
\end{align*}
\end{prop}
\noindent
For the proof, we need two lemmas. 
\begin{lem}\label{lem3.7}
The identity
\begin{align}\label{eqlem3.7} 
Y_p=\sum\limits_{k=p}^r(-q_1x_1)^{k-p}X_k+\frac{(-q_1 x_1)^{r+1-p}}{(1-q_1x_1)^r}
\end{align}
holds for $p=2,\ldots,r$, where $X_j$'s and $Yj$'s are same as those in Proposition \ref{lem10}. 
\end{lem}
\begin{proof}
For any $2\leq p \leq r$,
\begin{align*}
X_p=& \sum_{i=p}^r\left[\left({i-2\atop p-2}\right)+\left({i-2\atop
p-1}\right)\right](q_1x_1)^{i-p}
\frac{x_1^{r+1-i}x_{i+1}-x_{r+2}}{x_{r+2}}+(1-q_1x_1)^{-p}\\
=&Y_p-(1-q_1x_1)^{1-p}+(1-q_1x_1)^{-p}+q_1x_1\sum\limits_{i=p+1}^r\left({i-2\atop p-1}\right)(q_1x_1)^{i-p-1}
\frac{x_1^{r+1-i}x_{i+1}-x_{r+2}}{x_{r+2}}\\
=&Y_p-(1-q_1x_1)^{1-p}+(1-q_1x_1)^{-p}+q_1x_1Y_{p+1}-(1-\delta_{p,r})q_1x_1(1-q_1x_1)^{-p}\\
=&Y_p+q_1x_1Y_{p+1}+\delta_{p,r}q_1x_1(1-q_1x_1)^{-p},
\end{align*}
where $Y_{r+1}:=0$. Rewriting these formulas in matrix forms, we get
$$\begin{pmatrix}
X_2\\
X_3\\
\vdots\\
X_r
\end{pmatrix}=
\begin{pmatrix}
1 & q_1x_1\\
& 1 & q_1x_1\\
&& \ddots & \ddots\\
&&& 1 & q_1x_1\\
&&&& 1
\end{pmatrix}
\begin{pmatrix}
Y_2\\
Y_3\\
\vdots\\
Y_r
\end{pmatrix}+
\begin{pmatrix}
0\\
\vdots\\
0\\
\frac{q_1x_1}{(1-q_1x_1)^r}
\end{pmatrix}.$$
Hence
$$\begin{pmatrix}
Y_2\\
Y_3\\
\vdots\\
Y_r
\end{pmatrix}=\begin{pmatrix}
1 & q_1x_1\\
& 1 & q_1x_1\\
&& \ddots & \ddots\\
&&& 1 & q_1x_1\\
&&&& 1
\end{pmatrix}^{-1}
\begin{pmatrix}
X_2\\
\vdots\\
X_{r-1}\\
X_r-\frac{q_1x_1}{(1-q_1x_1)^r}
\end{pmatrix}.$$

It is known that if we set
$$(d_{ij})=
\begin{pmatrix}
1 & q_1x_1\\
& 1 & q_1x_1\\
&& \ddots & \ddots\\
&&& 1 & q_1x_1\\
&&&&1
\end{pmatrix}^{-1},$$
then
$$d_{ij}=\left\{
\begin{array}{ll}
(-q_1x_1)^{j-i}, & i\leq j,\\
0, & i>j.
\end{array}\right.$$
Thus we obtain the identity. 
\end{proof}
\begin{lem}\label{lem12} We have the following identities. 
\begin{itemize}
\item[(i)] 
\begin{align*}
\widetilde{A^{(0)}_j} & = 
\sum_{m=j}^{r-1} \left\{ 
\sum_{k=r-m+1}^{r+1} \left( {k-2 \choose r-m}  + \f{q_1 u_1}{1 + q_1 u_1} {k-2 \choose r-m-1} \right)\right. \\
& \quad \times \left.\vbox to 21pt{}(-q_1)^{k-r+m-2}
u_1^{k-r-2} \left( u_1^{r-k+2} u_k - u_{r+2} \right) \right\}
S_q(m+1, j+1)\\
& \quad + \left( \f{u_1}{1+q_1u_1} \sum_{k=2}^{r+1} (-q_1)^{k-2} u_k + \f{(-q_1)^r u_1 u_{r+2}}{(1+q_1 u_1)^2} \right)
S_q(r, j+1) 
\end{align*}
for $0  \leq j \leq r-1$, where $\widetilde{A_j^{(0)}}$'s are same as those in Theorem \ref{thm9}.   
\item[(ii)] 
$$X_p(1-q_1x_1)-Y_p = \f{u_1^{r+1-p}(1+q_1 u_1)^{p-1}}{u_{r+2}} \sum_{k=p+1}^{r+1} (-q_1)^{k-p} (u_1 u_k - u_{k+1})$$
for $2 \leq p \leq r$, where $X_p$'s and $Y_p$'s are same as those in Proposition \ref{lem10}. 
\item[(iii)]  
\begin{align*} 
& x_{r+2} \sum_{p=2}^{r} \left\{ X_p (1-q_1x_1) -Y_p \right\} \f{c_{p-1, j+1}}{x_1^{r+1-p}}\\
& = \sum_{m=j}^{r-1} \left\{ \sum_{k=r-m+1}^{r+1} \left( { k-2 \choose r-m} + \f{q_1 u_1}{1+q_1 u_1}{k-2 \choose r-m-1} \right) \right. \\
& \quad \left. \times \f{(-q_1)^{k-r+m-1}}{1+q_1u_1} (u_1 u_k - u_{k+1}) \right\} S_q(m+1, j+1) \\
& \quad + \f{u_1}{(1+q_1 u_1)^2} \sum_{k=2}^{r+1}(-q_1)^{k-1} (u_1 u_k - u_{k+1}) S_q(r, j+1) 
\end{align*}
for $0 \leq j \leq r-1$, where $c_{i, j}$'s are same as those in Lemma \ref{lem11}. 
\end{itemize}
\end{lem}
\begin{proof}
(i) We compute $\widetilde{A^{(0)}_j}$, which
by definition is
$$\widetilde{A^{(0)}_j} = \sum_{m=j}^{r-1} (x_{r+2-m}- x_1 x_{r+1-m}) S_q(m+1, j+1) + x_1 x_2 S_q(r, j+1).$$

By \eqref{x1} and \eqref{xj}, we have
\begin{align*}
x_1x_2=&\frac{u_1}{1+q_1u_1}\frac{1}{u_1^r}\left\{\sum\limits_{k=2}^{r+1}(-q_1u_1)^{k-2}
(u_1^{r+2-k}u_k-u_{r+2})+\frac{u_{r+2}}{1+q_1u_1}\right\}\\
=&\frac{u_1}{1+q_1u_1}\sum\limits_{k=2}^{r+1}(-q_1)^{k-2}u_k+\frac{(-q_1)^ru_1u_{r+2}}{(1+q_1u_1)^2}, 
\end{align*}
and
\begin{align*}
& x_{r+2-m}-x_1x_{r+1-m}\\
=& \frac{1}{u_1^m}\left\{\sum\limits_{k=r-m+2}^{r+1}\left({k-2\atop
r-m}\right)(-q_1u_1)^{k-r+m-2}(u_1^{r-k+2}u_k-u_{r+2})+\frac{u_{r+2}}{(1+q_1u_1)^{r-m+1}}\right\}\\
&-\frac{u_1}{1+q_1u_1}
\frac{1}{u_1^{m+1}}\left\{\sum\limits_{k=r-m+1}^{r+1}\left({k-2\atop
r-m-1}\right)(-q_1u_1)^{k-r+m-1}(u_1^{r-k+2}u_k-u_{r+2})+\frac{u_{r+2}}{(1+q_1u_1)^{r-m}}\right\}\\
=&\sum\limits_{k=r-m+1}^{r+1}\left\{\left({k-2\atop
r-m}\right)+\frac{q_1u_1}{1+q_1u_1}\left({k-2\atop
r-m-1}\right)\right\}(-q_1)^{k-r+m-2}u_1^{k-r-2}(u_1^{r-k+2}u_k-u_{r+2}).
\end{align*}
%
Hence we obtain the result. \\
(ii) By \eqref{eqlem3.4}, \eqref{x1} and \eqref{xj}, 
$$X_p=\frac{x_1^{r+1-j}}{x_{r+2}}u_{p+1}=\frac{u_1^{r+1-p}u_{p+1}(1+q_1u_1)^p}{u_{r+2}}.$$
Using \eqref{x1} and \eqref{eqlem3.7}, we have
\begin{align*}
Y_p=&\sum\limits_{k=p}^r\left(\frac{-q_1u_1}{1+q_1u_1}\right)^{k-p}\frac{u_1^{r+1-k}u_{k+1}(1+q_1u_1)^k}{u_{r+2}}
+\left(\frac{-q_1u_1}{1+q_1u_1}\right)^{r+1-p}(1+q_1u_1)^r\\
=&\frac{u_1^{r+1-p}(1+q_1u_1)^{p-1}}{u_{r+2}}\left\{\sum\limits_{k=p}^r(-q_1)^{k-p}u_{k+1}(1+q_1u_1)
+(-q_1)^{r+1-p}u_{r+2}\right\}.
\end{align*}
Hence we have 
\begin{align*}
& X_p(1-q_1x_1)-Y_p =  \frac{X_p}{1+q_1u_1}-Y_p\\
& =\frac{u_1^{r+1-p}u_{p+1}(1+q_1u_1)^{p-1}}{u_{r+2}} \\
&\quad - \frac{u_1^{r+1-p}(1+q_1u_1)^{p-1}}{u_{r+2}}\left\{\sum\limits_{k=p}^r(-q_1)^{k-p}u_{k+1}(1+q_1u_1)
+(-q_1)^{r+1-p}u_{r+2}\right\}.
\end{align*} 
Thus
\begin{align*}
&(X_p(1-q_1x_1)-Y_p)\frac{u_{r+2}}{u_1^{r+1-p}(1+q_1u_1)^{p-1}}\\
=& u_{p+1}-\sum\limits_{k=p}^r(-q_1)^{k-p}u_{k+1}(1+q_1u_1)
-(-q_1)^{r+1-p}u_{r+2}\\
=&-\sum\limits_{k=p+1}^r(-q_1)^{k-p}u_{k+1}+\sum\limits_{k=p}^r(-q_1)^{k+1-p}u_1u_{k+1}-(-q_1)^{r+1-p}u_{r+2}\\
=& \sum\limits_{k=p+1}^{r+1}(-q_1)^{k-p}(u_1u_k-u_{k+1}).
\end{align*}
%
(iii) By (ii), \eqref{x1} and \eqref{xj}, we have 
\begin{align*}
& x_{r+2} \sum_{p=2}^{r} \left( X_p(1-q_1 x_1) - Y_p \right) \f{c_{p-1, j+1}}{x_1^{r+1-p}}\\
=&\frac{u_{r+2}}{(1+q_1u_1)^{r+1}}\sum\limits_{p=2}^r\frac{u_1^{r+1-p}(1+q_1u_1)^{p-1}}{u_{r+2}}
\sum\limits_{k=p+1}^{r+1}(-q_1)^{k-p}(u_1u_k-u_{k+1})c_{p-1,j+1}\left(\frac{1+q_1u_1}{u_1}\right)^{r+1-p}\\
=&\frac{1}{1+q_1u_1}\sum\limits_{p=2}^r\sum\limits_{k=p+1}^{r+1}(-q_1)^{k-p}(u_1u_k-u_{k+1})c_{p-1,j+1}.
\end{align*}

By definition of $c_{i, j}$ in Lemma \ref{lem11}: 
$$c_{p-1,j+1}=\sum\limits_{m=1}^{\min(p-1,r-j)}\left({p-2\atop m-1}\right)(-q_1)^{p-1-m}a_{m,j+1},$$
we have
\begin{align*}
& x_{r+2} \sum_{p=2}^{r} \left( X_p(1-q_1 x_1) - Y_p \right) \f{c_{p-1, j+1}}{x_1^{r+1-p}}\\
=&\frac{1}{1+q_1u_1}\sum\limits_{p=2}^r\sum\limits_{k=p+1}^{r+1}(-q_1)^{k}(u_1u_k-u_{k+1})
\sum\limits_{m=1}^{\min(p-1,r-j)}\left({p-2\atop
m-1}\right)(-q_1)^{-m-1}a_{m,j+1}.
\end{align*}
Changing the orders of the sums, we have
\begin{align*}
& x_{r+2} \sum_{p=2}^{r} \left( X_p (1-q_1 x_1) - Y_p \right) \f{c_{p-1, j+1}}{x_1^{r+1-p}}\\
=&\frac{1}{1+q_1u_1}\sum\limits_{m=1}^{r-j}\sum\limits_{p=m+1}^{r}\sum\limits_{k=p+1}^{r+1}
(-q_1)^{k-m-1}(u_1u_k-u_{k+1})\left({p-2\atop
m-1}\right)a_{m,j+1}\\
=&\frac{1}{1+q_1u_1}\sum\limits_{m=1}^{r-j}\sum\limits_{k=m+2}^{r+1}\left\{\sum\limits_{p=m+1}^{k-1}
\left({p-2\atop m-1}\right)\right\}
(-q_1)^{k-m-1}(u_1u_k-u_{k+1})a_{m,j+1}\\
=&\frac{1}{1+q_1u_1}\sum\limits_{m=1}^{r-j}\sum\limits_{k=m+2}^{r+1}
\left({k-2\atop m}\right)(-q_1)^{k-m-1}(u_1u_k-u_{k+1})a_{m,j+1}.
\end{align*}

Using \eqref{x1} and the definition of $a_{i, j}$ in Lemma \ref{lem11}: 
$$a_{m, j+1} = S_q(r+1-m, j+1) -x_1 S_q(r-m, j+1), $$
the right-hand side is calculated as  
\begin{align*}
& \f{1}{1+q_1 u_1} \sum_{m=1}^{r-j} \sum_{k=m+2}^{r+1} {k-2 \choose m} (-q_1)^{k-m-1} (u_1 u_k - u_{k+1}) S_q(r+1-m, j+1) \\
& - \f{u_1}{( 1+q_1 u_1 )^2} \sum_{m=1}^{r-j} \sum_{k=m+2}^{r+1} {k-2 \choose m} (-q_1)^{k-m-1} (u_1 u_k - u_{k+1})S_q(r-m, j+1) \\
& =  \f{1}{1+q_1 u_1} \sum_{m=j}^{r-1} \sum_{k=r-m+1}^{r+1} {k-2 \choose r-m} (-q_1)^{k-r+m-1} (u_1 u_k - u_{k+1}) S_q(m+1, j+1) \\
&  + \f{q_1 u_1}{( 1+q_1 u_1 )^2} \sum_{m=j}^{r-2} \sum_{k=r-m+1}^{r+1} {k-2 \choose r-m-1} (-q_1)^{k-r+m-1} (u_1 u_k - u_{k+1})S_q(m+1, j+1)\\
& =  \sum_{m=j}^{r-1}  \left\{ \sum_{k=r-m+1}^{r+1} \left(  {k-2 \choose r-m} + \f{q_1 u_1}{1+q_1 u_1 } {k-2 \choose r-m-1}\right) 
\f{(-q_1)^{k-r+m-1}}{1+q_1 u_1} (u_1 u_k - u_{k+1}) \right\}\\
& \times S_q(m+1, j+1) - \f{q_1 u_1}{( 1+q_1 u_1 )^2} \sum_{k=2}^{r+1} (-q_1)^{k-2} (u_1 u_k - u_{k+1})S_q(r, j+1). 
\end{align*}
\end{proof}
\begin{proof}[Proof of Proposition \ref{propAj}]
By Lemma \ref{lem12},  
$A_j$ in Proposition \ref{Prop} is expressed as required. 
\end{proof}
\subsection{The constant $c$}
Finally, we write the constant $c$ in Proposition \ref{Prop} in terms of $u_j$'s. 
As before, put $q_1 = 1- q$.   
\begin{prop}\label{lem13}
We have
$$c= 1-\f{u_1}{1+q_1u_1} - t \f{1-q u_1}{1 + q_1 u_1} \sum_{k=2}^{r+1} q^{k-2} u_k - t \f{q^r}{1+q_1u_1}u_{r+2}. $$
\end{prop}
\begin{proof}
Recall that 
$$c = 1-(x_1+tx_2) - t\sum_{i=0}^{r-1} (x_{r+2-i} - x_1 x_{r+1-i}). $$
By \eqref{x1} and \eqref{xj}, we have 
$x_1 = \f{u_1}{1+q_1 u_1}$, $x_2= \sum_{k=2}^{r+1} (-q_1)^{k-2} u_k + \f{(-q_1)^r u_{r+2}}{1+q_1 u_1}$, and
\begin{align*}
& x_{r+2-i} - x_i x_{r+1-i} \\
& = \sum_{k=r-i+1}^{r+1} \left\{ {k-2 \choose r-i} + \f{q_1 u_1}{1+q_1 u_1} {k-2 \choose r-i-1} \right\} (-q_1)^{k-r+i-2} (u_k - u_1^{k-r-2} u_{r+2}).   
\end{align*}
We have  
\begin{align*}
& 
\sum_{i=0}^{r-1} (x_{r+2-i} - x_i x_{r+1-i}) \\
& = 
\sum_{k=2}^{r+1}\sum_{i=r+1-k}^{r-1} \left\{ {k-2 \choose r-i} + \f{q_1 u_1}{1+q_1 u_1} {k-2 \choose r-i-1} \right\} (-q_1)^{k-r+i-2} (u_k - u_1^{k-r-2} u_{r+2})\\
& = 
\sum_{k=2}^{r+1}\left\{ 
\sum_{i=1}^{k-1} {k-2 \choose i}(-q_1)^{k-2-i} + \f{q_1 u_1}{1+q_1 u_1} \sum_{i=0}^{k-2}{k-2 \choose i}(-q_1)^{k-i-3} \right\} 
(u_k - u_1^{k-r-2} u_{r+2})\\
& = 
\sum_{k=2}^{r+1}\left\{ 
 \f{1-q u_1}{1+q_1 u_1} q^{k-2} - (-q_1)^{k-2}\right\}  (u_k - u_1^{k-r-2} u_{r+2}).
\end{align*}
Thus 
\begin{align*}
& -x_2 - \sum_{i=0}^{r-1} (x_{r+2-i} - x_i x_{r+1-i}) \\
& = - \f{(-q_1)^r u_{r+2}}{1+q_1 u_1} - \f{1-q u_1}{1+q_1 u_1}\sum_{k=2}^{r+1} q^{k-2} u_k
+ \left( \f{1-q u_1}{1+q_1 u_1} \sum_{k=2}^{r+1} q^{k-2} u_1^{k-2}
 - \sum_{k=2}^{r+1}(-q_1)^{k-2} u_1^{k-2}  \right)u_1^{-r}u_{r+2} \\
& = - \f{(-q_1)^r u_{r+2}}{1+q_1 u_1} - \f{1-q u_1}{1+q_1 u_1}\sum_{k=2}^{r+1} q^{k-2} u_k +\f{-q^r+(-q_1)^r }{1+q_1 u_1}  u_{r+2} \\
& = - \f{1-q u_1}{1+q_1 u_1}\sum_{k=2}^{r+1} q^{k-2} u_k -\f{q^r}{1+q_1 u_1}  u_{r+2} . 
\end{align*}
This completes the proof.  
\end{proof}
\subsubsection*{Acknowledgments}
Thanks are due to Tatsushi Tanaka for his helpful  comments and advice. 
The first author is supported by the National Natural Science Foundation of China No. 11471245. 
The second author is supported by the Grant-in-Aid for Young Scientists (B) No. 15K17523, Japan Society for the Promotion of Science. 
%

\end{document}